\documentclass[12pt]{amsart}%
\usepackage{verbatim}
\usepackage{latexsym}
\usepackage{amsmath}
\usepackage{amsfonts}
\usepackage{tabularx}
\usepackage{amssymb,hyperref}
\usepackage[normalem]{ulem}
\usepackage{graphicx}
\usepackage{caption}
\usepackage{subcaption}

\usepackage{dsfont}
\newcommand{\1}{\mathds{1}}

\newcommand{\D}{\mathbb{D}}

\newcommand{\C}{\mathbb{C}}

\newcommand{\R}{\mathbb{R}}

\newcommand{\Z}{\mathbb{Z}}
\newcommand{\PP}{\mathbb{P}}
\newcommand{\TT}{\mathbb{T}}

\newcommand{\E}{\mathbb{E}}

\newtheorem{thm}{Theorem}[section]

\newtheorem{prop}[thm]{Proposition}
\newtheorem{coro}[thm]{Corollary}
\newtheorem{lem}[thm]{Lemma}

\newtheorem{rem}[thm]{Remark}
\newtheorem{ex}[thm]{Example}

\numberwithin{equation}{section}
\numberwithin{table}{section}
\numberwithin{figure}{section}

\newcommand{\dd}{\;\mathrm{d}}
\newcommand{\e}{\mathrm{e}}
\newcommand{\m}{\mathrm{m}}

\newcommand{\Li}{\mathrm{Li}} 
\newcommand{\re}{\mathop{\mathrm{Re}}} 
\newcommand{\im}{\mathop{\mathrm{Im}}} 
\newcommand{\hypgeo}[2]{%
  {\vphantom{F}}_{#1}\kern-\scriptspace F_{#2}%
}

\usepackage{xcolor}
\newcommand{\mcom}[1]{{\color{blue}{Matilde: #1}} }

\newcommand{\kommentar}[1]{}
\kommentar{
\newcommand*\pFqskip{8mu}
\catcode`,\active
\newcommand*\pFq{\begingroup
        \catcode`\,\active
        \def ,{\mskip\pFqskip\relax}%
        \dopFq
}
\catcode`\,12
\def\dopFq#1#2#3#4#5{%
        {}_{#1}F_{#2}\left(\left.\genfrac..{0pt}{}{#3}{#4}\right|#5\right)%
        \endgroup
}}

\newmuskip\pFqmuskip
\newcommand*\pFq[6][8]{%
  \begingroup 
  \pFqmuskip=#1mu\relax
  \mathchardef\normalcomma=\mathcode`,
  \mathcode`\,=\string"8000
  \begingroup\lccode`\~=`\,
  \lowercase{\endgroup\let~}\pFqcomma
  {}_{#2}F_{#3}{\left(\!\left.\genfrac..{0pt}{}{#4}{#5}\right|#6\!\right)}%
  \endgroup
}
\newcommand{\pFqcomma}{{\normalcomma}\mskip\pFqmuskip}

\begin{document}

\title{Random walks through the areal Mahler measure: steps in the complex plane. }

\author{Matilde N. Lal\'in}
\author{Siva Sankar Nair}
\author{Berend Ringeling}
\author{Subham Roy}

\address{Matilde Lal\'in:  D\'epartement de math\'ematiques et de statistique, Universit\'e de Montr\'eal. CP 6128, succ. Centre-ville. Montreal, QC H3C 3J7, Canada}\email{matilde.lalin@umontreal.ca}
\address{Siva Sankar Nair: D\'epartement de math\'ematiques et de statistique, Universit\'e de Montr\'eal. CP 6128, succ. Centre-ville. Montreal, QC H3C 3J7, Canada}\email{siva.sankar.nair@umontreal.ca}
\address{Berend Ringeling:  D\'epartement de math\'ematiques et de statistique, Universit\'e de Montr\'eal. CP 6128, succ. Centre-ville. Montreal, QC H3C 3J7, Canada}\email{bjringeling@gmail.com}
\address{Subham Roy: Charles University, Faculty of Mathematics and Physics, Department of Algebra, Sokolovsk\'a 83, 186 00 Praha 8, Czech Republic} \email{subham.roy@matfyz.cuni.cz}



\begin{abstract}
    We study the areal Mahler measure of the two-variable, $k$-parameter family $x+y+k$ 
    and prove explicit formulas that demonstrate its relation to the standard Mahler measure of these polynomials. The proofs involve interpreting the areal Mahler measure as a random walk in the complex plane and utilizing the areal analogue of the Zeta Mahler function to arrive at the result. Using similar techniques, we also present formulas for a three-variable family $(x+1)(y+1)+kz$ in terms of the standard Mahler measure, along with terms that involve certain hypergeometric functions. For both families we show that its areal Mahler measure is, up to elementary functions, a linear combination of the normal Mahler measure and the volume of the Deninger cycle of the corresponding family.

\end{abstract}

\maketitle

\section{Introduction}
The (logarithmic) Mahler measure of a non-zero rational function $P \in \C(x_1,\dots,x_n)^\times$ is a height function defined as 
\begin{equation*}
 \m(P):=\frac{1}{(2\pi i)^n}\int_{\mathbb{T}^n}\log|P(x_1,\dots, x_n)|\frac{\dd x_1}{x_1}\cdots \frac{\dd  x_n}{x_n},
\end{equation*}
where the integration is taken with respect to the Haar measure on the $n$-dimensional unit torus  $\mathbb{T}^n=\{(x_1,\dots,x_n)\in \mathbb{C}^n\, :\, |x_1|=\cdots=|x_n|=1\}$. 

 When  $P$ is a univariate polynomial, Jensen's formula implies that $\m(P)$ can be expressed in terms of the roots of $P$ that lie outside the unit circle in the complex plane. In the multivariable case, the Mahler measure remains a  mysterious and intriguing quantity with no general formula to compute it. However, there are various instances where $\m(P)$ is known to yield special values of functions that have arithmetic significance, such as the Riemann zeta function, $L$-functions, and so on. One of the first formulas of this type was given by Smyth \cite{S1,B1}:
\begin{equation}\label{eq:Smyth}
\m(x+y+1) = \frac{3 \sqrt{3}}{4 \pi} L(\chi_{-3},2),
\end{equation}
where  $L(\chi_{-3},s)$ is the Dirichlet $L$-function associated to the primitive character $\chi_{-3}$ of conductor 3. The connection with special values of $L$-functions has been explained in terms of evaluations of regulators and Beilinson's conjectures by Deninger \cite{Deninger}, Boyd \cite{Bo98}, and Rodriguez-Villegas \cite{RV}. It is also observed that Mahler measure computations of certain families of polynomials yield values of hypergeometric functions. This can often be explained by the fact that these families of Mahler measures satisfy certain hypergeometric type differential equations  (see \cite[Chapter 5]{BrunaultZudilin} for more details). The initial connections were discovered by Rodriguez-Villegas \cite{RV}. For example, for $k>0$,
\begin{equation}\label{eq:FRV}\m\left(x+\frac{1}{x}+y+\frac{1}{y}+k\right)=\log (k)- \frac{2}{k^2}\re\pFq{4}{3}{1,1,\frac{3}{2},\frac{3}{2}}{2,2,2}{\frac{16}{k^2}}.\end{equation}
(See also \cite[Proposition 5.2]{BrunaultZudilin}.) The connection to hypergeometric functions can be exploited in favorable cases to relate the Mahler measures to special values of $L$-functions as conjectured by Boyd \cite{Bo98}. Rogers and Zudilin \cite{RZ12,RZ14} proved several of these initial formulas, such as 
\[\m\left(x+\frac{1}{x}+y+\frac{1}{y}+1\right)=L'(E,0),\]
where $E$ is an elliptic curve of conductor $15$ corresponding to the projective closure of $x+\frac{1}{x}+y+\frac{1}{y}+1=0$.

The appearance of such special values in Mahler measure evaluations sparked great interest in their computations and has led to the study of several variations of the Mahler measure. The generalized Mahler measure \cite{GonOyanagi}, multiple and higher Mahler measures \cite{KLO}, and the Zeta Mahler function \cite{Akatsuka} are some examples that yield similar computations. In this article, we consider the (logarithmic) areal Mahler measure of $P\in\C(x_1,\dots,x_n)^\times$ as defined by Pritsker \cite{Pritsker}:
\begin{equation*}
\m_\D(P)=\frac1{\pi^n}\int_{\D^n}\log|P(x_1,\dots,x_n)| \dd A(x_1)\cdots \dd A(x_n),
\end{equation*}
where
\[
\D^n=\{(x_1,\dots,x_n)\in\C^n:|x_1|,\dots,|x_n|\le1\}
\]
is the product of $n$ unit disks, and the measure is the natural measure in the $A^0$ Bergman space. Recently, Lalín and Roy \cite{Lalin-Roy} obtained interesting expressions for certain two-variable polynomials such as
\begin{equation}\label{eq:k=1}
\m_\D(x+y+1)=\frac{3\sqrt3}{4\pi}L(\chi_{-3},2)+\frac16-\frac{11\sqrt3}{16\pi},
\end{equation}
and
\begin{equation}\label{eq:k=sqrt2}
\m_\D(x+y+\sqrt2)=\frac{L(\chi_{-4},2)}{\pi}+\mathcal{C}_{\sqrt2}+\frac38-\frac3{2\pi},
\end{equation}
where $\chi_{-4}$ denotes the primitive Dirichlet character of conductor 4 and $\mathcal{C}_{\sqrt2}$ is a constant that was expressed in terms of values of a generalized hypergeometric function and the Gamma function in \cite{Lalin-Roy}. As we will later see in Theorem \ref{thm:x+y+k}, $\mathcal{C}_{\sqrt2}=\frac{\log 2}{4}$.

These results are particularly fascinating since we may compare formula \eqref{eq:k=1} with the standard Mahler measure formula \eqref{eq:Smyth} and observe the same term involving the $L$-function, along with some extra terms. Similarly, it is known that 
\begin{equation}\label{eq:k=sqrt2classical}
\m(x+y+\sqrt2)=\frac{L(\chi_{-4},2)}{\pi}+\frac{\log2}4,
\end{equation}
and we observe the same phenomenon comparing with formula \eqref{eq:k=sqrt2}. In this article, we make this observation more precise and prove the following general statement.
\begin{thm}\label{thm:x+y+k} 
If $0\leq k \leq 2$, we have 
   \begin{equation}\label{eq:k<2}
   \m_{\D}(x+y+k)=\m(x+y+k)-\frac{k\sqrt{4-k^2}(10+k^2) +(8-16k^2)\arccos\left(\frac{k}{2}\right)}{16\pi }.\end{equation}
If $k\geq 2$, then 
\begin{equation}\label{eq:k>2}
\m_{\D}(x+y+k)=\m(x+y+k)=\log (k).
\end{equation}
\end{thm}
\begin{rem}\label{rem:CM}
The value of $\m(x+y+k)$ for $0\le k\le 2$ is known due to a result of Cassaigne and Maillot (see Theorem \ref{thm:CM}), which gives us
\[
 \m(x + y +k) = \frac{1}{\pi} \left( 2 \log (k) \arcsin\left(\frac{k}{2}\right) + D \left( \e^{2 i \arcsin\left(\frac{k}{2}\right)}\right)\right).
\]
Here $\arcsin\left(\frac{k}{2}\right), \arccos\left(\frac{k}{2}\right) \in [0, \pi),$ and $D(z)$ denotes the Bloch--Wigner dilogarithm (see equation \eqref{eq:BWdilog}). The dilogarithm terms above often yield values of certain $L$-functions, such as the Riemann zeta function when $k=1$, and the Dirichlet $L$-function associated to the character $\chi_{-4}$ when $k=\sqrt2$.
\end{rem}
We exhibit two proofs of Theorem \ref{thm:x+y+k}. The first proof consists of evaluating the areal Zeta Mahler function (see definition \eqref{eq:ZAMMdefn1}) by interpreting it as the expected value of a random walk $|X + Y + k|$, where $X$ and $Y$ are random variables whose radii and angles are uniformly distributed on $[0, 1]$ and on $[0, 2 \pi)$, respectively. Then the areal Mahler measure is recovered by differentiating this function at $s=0$. The second proof follows similar ideas, but requires a known areal Zeta Mahler function. 
While the first proof is more involved, it gives access to formulas for the areal version of the higher Mahler measure (see Theorem \ref{thm:ZAMMk+x+y}).

We also prove a similar result involving the difference of the areal Mahler measure and the standard Mahler measure, but of a three-variable polynomial family given by
\[
Q_k(x,y,z)=(x+1)(y+1)+kz,
\] 
for $k$ a non-negative real number.  The formula for $k=1$,
\begin{equation}\label{eq:Boyd}
\m((1 + x)(1 + y)+z) = -2L'(E, -1),\end{equation}
where $E : (1 + x)(1 + y)(1 + \frac{1}{x} )(1 + \frac{1}{y} ) = 1$ 
is an elliptic curve of conductor 15,
was conjectured by Boyd and Rodriguez-Villegas in \cite{Boyd-conj}, shown to be related to the corresponding $L$-function in \cite{Lalin-Crelle}, and finally completely proven by Brunault in \cite{Brunault-1x1yz} using the  computations of Goncharov regulator integrals associated to $K_4$ classes on modular curves developed by Brunault and Zudilin in \cite{BrunaultZudilin-modularregulators}. 

For general $k$, the Mahler measure of $Q_k$ is not expected to be expressible directly as a sum of hypergeometric functions. As the next result shows, we expect the same for the areal Mahler measure.


\begin{thm} \label{thm:Qk}
Let $F(t)$ and $G(t)$ denote the following functions written in terms of the Gauss hypergeometric function
\[
F(t) = \frac{1}{2 \pi} \cdot \pFq{2}{1}{ \frac{1}{2}, \frac{1}{2}}{1}{1 - \frac{t^2}{16}},
\]
and
\[
G(t)=\frac{1}{2 \pi }\cdot\pFq{2}{1}{\frac{1}{2}, -\frac{1}{2}}{1}{1-\frac{t^2}{16}}.
\]
    Then, for $0<k<4$, we have
    \begin{align}
    \m_{\D}(Q_k) - \frac{k^2 + 8}{8} \m (Q_k)=&\;\frac{9}{8k^2} - \frac{1}{2} - \frac{k^2}{8}\log(k) + c_0(k)\left(\frac{1}{2} - \frac{9}{8 k^2} + \frac{5 k^2}{32} \right)\label{eq:34}\\
    &\;+ F(k) \left(-\frac{9}{8 k} - \frac{29 k}{64} + \frac{17 k^3}{128} \right) + G(k) \left(-\frac{9}{8 k} - \frac{49 k}{32}\right),\nonumber
    \end{align}
    where $c_0(k)$ is defined as
\[
c_0(k) = \int_{k}^4 F(t) \dd t.
\]
   If $k\ge4$, we have
\begin{equation*}
    \m_{\D}(Q_k) - \frac{k^2 + 8}{8} \m (Q_k)=\frac{9}{8k^2} - \frac{1}{2} - \frac{k^2}{8} \log(k).
    \end{equation*}
\end{thm}
\begin{rem}Note that for $k\in\C$, we have $\m_{\D}(x+y+k)=\m_{\D}(x+y+|k|)$, $\m_{\D}(Q_k)=\m_{\D}(Q_{|k|})$ and similarly with the respective classical Mahler measures. Thus, both Theorem \ref{thm:x+y+k} and Theorem \ref{thm:Qk}  cover all possible complex values of $k$. \end{rem}
\begin{rem} As we will later see, the terms $\arccos\left(\frac{k}{2}\right)$ and $c_0(k)$ appearing in Theorem \ref{thm:x+y+k} and Theorem \ref{thm:Qk} can be interpreted as volumes of Deninger cycles that arise naturally in the Mahler measure computations of these families. 
\end{rem}

The proof of Theorem \ref{thm:Qk} follows on similar lines as the second proof of Theorem \ref{thm:x+y+k}. The argument going through the computation of the areal Zeta Mahler function seems to lead to complicated expressions involving Meijer $G$-functions, and we decided not to pursue this direction, since the main focus of our results is the areal Mahler measure. 

The formulas appearing in Theorems \ref{thm:x+y+k} and \ref{thm:Qk}, and the formulas for the classical Mahler measure (Remark \ref{rem:CM} and Lemma \ref{Lem:ordinaryMM}) also allow for a rapid numerical computation of the areal Mahler measure. In contrast, directly computing them from the definition would involve evaluating a four-fold and a six-fold integral, respectively. For example, we find by using PARI/GP \cite{PARI}
\begin{equation}\label{eq:arealBoyd}
\m_{\D}(Q_1) = 0.181650509823419975804057948562341831685\ldots.
\end{equation}
Numerical computations are central to finding conjectures relating Mahler measures and their variants to special values of $L$-functions and other arithmetically significant functions.  For example, an evaluation like \eqref{eq:arealBoyd} could lead to interesting identities such as equation \eqref{eq:Boyd}.


The article is organized as follows. In Section \ref{sec:bg}, we first present some results and definitions that are required in preparation for subsequent sections. Section \ref{sec:x+y+k} concerns the two proofs of Theorem \ref{thm:x+y+k}. We begin by realizing the areal Zeta Mahler function as a solution to a hypergeometric differential equation in Subsection \ref{sec:azmmkxy}, following which the areal Mahler measure is obtained in Subsection \ref{sec:ammkxy}. The second proof, involving a direct computation of the integral, is described in Subsection \ref{sec:direct}. We prove Theorem \ref{thm:Qk} in Section \ref{sec:Qk} using a similar direct computation but with a more complicated density function. Finally, we end with Section \ref{sec:vol}, where we provide further context to the terms appearing in Theorems \ref{thm:x+y+k} and \ref{thm:Qk}, by relating them to volumes of Deninger cycles and also providing a modular interpretation. 

\subsection*{Acknowledgements} The authors would like to thank Wadim Zudilin for helpful discussions.
This work has been supported by the Natural Sciences and Engineering Research Council of Canada, (RGPIN-2022-03651 to ML),  the Fonds de recherche du Qu\'ebec - Nature et technologies, (Projet de recherche en \'equipe 300951 to ML), the  Institut des sciences math\'ematiques (doctoral fellowships to SSN and SR, postdoctoral fellowship to BR),  the Centre de recherches math\'ematiques (postdoctoral fellowship to BR),   and the Czech Science Foundation GA\v{C}R, (grant 21-00420M to SR).

\section{Background}
\label{sec:bg}

In this section we collect some definitions and previous results that will be useful for the rest of the article. 

We start by giving some background to understand the appearance of $L$-functions in formulas such as \eqref{eq:Smyth}, \eqref{eq:k=1}, and \eqref{eq:k=sqrt2}. These special values arise from evaluations of the dilogarithm. This function is defined for $z\in \C$ such that $|z|<1$ by 
\[\Li_2(z)=\sum_{n=1}^\infty \frac{z^n}{n^2}.\]
It can be analytically continued to $\C\setminus [0,1)$ by 
\[\Li_2(z)=-\int_0^z \frac{\log(1-t)}{t} \dd t.\]
The Bloch--Wigner dilogarithm is given by 
\begin{equation}\label{eq:BWdilog}  D(z) := \im \Big(\Li_2(z) + \log(1-z) \log|z|\Big).\end{equation} 
It is a continuous function in $\PP^1(\C)$, which is real analytic in $\C\setminus \{0,1\}$. The following evaluations yield special values of Dirichlet $L$-functions
\[D\left(e^{\pi i/3}\right) =\frac{3\sqrt{3}}{4} L(\chi_{-3},2),\qquad D(i)=L(\chi_{-4},2),\]
where $\chi_{-3}=\left(\frac{-3}{\cdot}\right)$ and $\chi_{-4}=\left(\frac{-4}{\cdot}\right)$ are the Dirichlet characters of conductors $3$ and $4$ respectively. 

An example of the connection between dilogarithms and Mahler measure is given by  Cassaigne and Maillot's formula.
\begin{thm} \cite{Maillot}\label{thm:CM}
Let $a,b,c \in \C^*$. We have that  
\[\pi \m(a+bx+cy) = \left \{ \begin{array}{ll} D\left(\left|\frac{a}{b}\right| e^{i \gamma}\right) + \alpha \log |a| + \beta \log |b| + \gamma \log |c| &
\triangle,\\ \\
\pi \log \max \{ |a|, |b|, |c|\} & \hspace{-0.6cm}\mbox{not}\: \triangle,
\end{array}\right.
\] 
where $\Delta$ denotes the property that  $|a|, |b|,$ and $|c|$ are the lengths of the sides of a triangle with angles $\alpha$, $\beta$ and $\gamma$ forming the respective opposite angles, and $D$ denotes the  Bloch--Wigner dilogarithm. 
\end{thm}

Frequently the Mahler measure appears as evaluations of hypergeometric functions, such as in equation \eqref{eq:FRV}.  The generalized hypergeometric function is 
defined by 
\begin{equation*}
 _pF_{q}\!\left(\left.{\begin{matrix}a_{1},\dots ,a_{p}\\b_{1},\dots ,b_{q}\end{matrix}}\;\right|\,z\right)= \sum_{n=0}^\infty \frac{(a_1)_n\cdots(a_p)_n}{(b_1)_n\cdots(b_q)_n} \, \frac {z^n} {n!},\end{equation*}
where $b_j\not\in\Z_{\leq 0}$ and where $(a)_n$ denotes the Pochhammer symbol given by 
 $(a)_0 = 1$, and for $n\geq 1$, 
\[
(a)_n = a(a+1)(a+2) \cdots (a+n-1).\]
We will often work with the condition $p=q+1$, in which case the series converges for $|z|<1$. In particular, when $p=2$ and $q= 1,$ it is called the Gauss hypergeometric function. This function has the following properties. 
 

\begin{thm}(Gauss Hypergeometric Theorem, Eq 15.1.20 in \cite{AS})\label{thm:Gauss} Let $a, b, c \in \C$ such that $c \not \in \Z_{\leq 0}$ and $\re(c-a-b)>0$. Then 
\[\pFq{2}{1}{a,b}{c}{1}=\frac{\Gamma(c)\Gamma(c-a-b)}{\Gamma(c-a)\Gamma(c-b)}.\]
\end{thm}
The next result gives an integral representation for $\pFq{2}{1}{a,b}{c}{z}$ in certain cases. 
\begin{thm}\cite[Theorem 2.2.1]{AAR}
\label{thm:hyperint} If $\re(c)>\re(b)>0$, then 
\[\pFq{2}{1}{a,b}{c}{z} =\frac{\Gamma(c)}{\Gamma(b)\Gamma(c-b)}\int_0^1 t^{b-1}(1-t)^{c-b-1}(1-zt)^{-a}\dd t,\]
in the $z$-plane cut along the real axis $[1,\infty)$. It is understood that $\arg(t)=\arg(1-t)=0$ and $(1-zt)^{-a}$ has its principal value. 
\end{thm}
We have that the hypergeometric function $\mathcal{W}(z)={} _pF_{q}\!\left(\left.{\begin{matrix}a_{1},\dots ,a_{p}\\b_{1},\dots ,b_{q}\end{matrix}}\;\right|\,z\right)$ satisfies the differential equation \begin{equation}\label{diffeq}
    z\prod_{j=1}^p \left(z\frac{\dd}{\dd z} + a_j\right)\mathcal{W}(z) = z\frac{\dd}{\dd z} \prod_{\ell=1}^q  \left(z\frac{\dd}{\dd z} + b_\ell - 1\right)\mathcal{W}(z).
\end{equation} (See \cite{Norlund}.)

The Meijer $G$-functions $G_{p \, q}^{m \, n}$ are generalizations of hypergeometric functions that may appear as solutions of certain hypergeometric style differential equations. They are given by  \begin{align}\label{eq:G}
 G_{p,q}^{\,m,n}\!\left(\left.{\begin{matrix}a_{1},\dots ,a_{p}\\b_{1},\dots ,b_{q}\end{matrix}}\;\right|\,z\right)={\frac {1}{2\pi i}}\int _{L}{\frac {\prod _{j=1}^{m}\Gamma (b_{j}-s)\prod _{j=1}^{n}\Gamma (1-a_{j}+s)}{\prod _{j=m+1}^{q}\Gamma (1-b_{j}+s)\prod _{j=n+1}^{p}\Gamma (a_{j}-s)}}\,z^{s}\dd s,
 \end{align}
 where the path of integration is defined in \cite[Section 16.17]{dlmf}. 

A key ingredient for the evaluation of the Mahler measure is Jensen's formula, which implies that, for a polynomial $P(x) = a \prod_{j = 1}^d (x - \alpha_j) \in \C[x]$,
\begin{equation}\label{eq:Jensen}
\m(P)=\log|a|+ \sum_{j = 1}^d \log^{+}|\alpha_j|.
    \end{equation}
The following result of Pritsker gives an equivalent to Jensen's formula for the areal Mahler measure.
\begin{thm}\cite[Theorem 1.1]{Pritsker}\label{thm:Pritsker}
For a polynomial $P(x) = a \prod_{j = 1}^d (x - \alpha_j) \in \C[x]$, then
\[\m_{\D}(P) =  \log|a| + \sum_{j = 1}^d \log^{+}|\alpha_j| + \frac{1}{2} \sum_{|\alpha_j| < 1} (|\alpha_j|^2 - 1).\]
\end{thm}

In certain cases, the Mahler measure and the areal Mahler measure have particularly simple expressions depending on the size of the coefficients. 
\begin{lem}\label{lem:Pritsker} \cite[Example 5.2,c]{Pritsker}
Let $P(x_1,\dots,x_n) \in \C[x_1,\dots, x_n]$ be a polynomial of total degree $d$ given by 
\[P(x_1,\dots,x_n)=\sum_{k_1+\cdots+k_n\leq d} a_{k_1,\dots,k_n}x_1^{k_1}\cdots x_n^{k_n}\]
with the property that 
\[|a_{0,\dots,0}| \geq \sum_{k_1+\cdots+k_n\leq d} |a_{k_1,\dots,k_n}|.\]
Then 
\[\m_{\D}(P)=\m(P)=\log|a_{0,\dots,0}|.\]
\end{lem}
A central function used to find formulas for the Mahler measure is the Zeta Mahler function. It was defined by Akatsuka \cite{Akatsuka} for a nonzero rational function $P\in \C(x_1,\dots,x_n)^\times$ as follows
\[ 
Z(s,P) 
:= \frac{1}{(2\pi i)^n} \int_{\TT^n} 
\left|P\left(x_1, \dots, x_n\right)\right|^s 
\frac{\dd x_1}{x_1}\cdots \frac{\dd x_n}{x_n},
\]
where $s$ is a complex variable in a neighborhood of $0$. A key aspect of the Zeta Mahler function is that it satisfies 
\[\left. \frac{\partial}{\partial s} Z(s,P)\right|_{s=0}=\m(P).\]
This object was further studied in \cite{KLO, Biswas, BiswasMurty, SasakiIJNT, Ringeling}. The following case was computed by Akatsuka.
\begin{thm}\cite[Theorem 2]{Akatsuka} \label{thm:Akatsuka} For $k\in \C$ such that $|k|\not =1$, we have that 
\[Z(s, x+k)=\begin{cases} 
\pFq{2}{1}{ -\frac{s}{2},-\frac{s}{2}}{1}{|k|^2}& |k|<1,\\ \\
|k|^s \pFq{2}{1}{ -\frac{s}{2},-\frac{s}{2}}{1}{|k|^{-2}}& |k|>1.
\end{cases}\]
\end{thm}
We also have the following formula due to Kurokawa, Lal\'in, and Ochiai \cite{KLO}:
\begin{equation}\label{eq:KLO}
Z(s,x+1)=\frac{\Gamma(1+s)}{\Gamma\left(1+\frac{s}{2}\right)^2}.\end{equation}

Lal\'in and Roy \cite{Lalin-Roy} gave an areal version of the Zeta Mahler function:
\begin{equation}\label{eq:ZAMMdefn1}
Z_\mathbb{D}(s,P) 
:= \frac{1}{\pi ^n} \int_{\mathbb{D}^n} 
\left|P\left(x_1, \dots, x_n\right)\right|^s 
\dd A(x_1)\dots \dd A(x_n),\end{equation}
which, when $P$ is a polynomial, converges absolutely and locally uniformly for $\re(s)>\sigma_1(P)$ for some $\sigma_1(P)< 0$ (see Propositions 5.2.1 and 5.2.2 in \cite{Subham-thesis}). 

The areal Zeta Mahler function satisfies that 
\[\left. \frac{\partial}{\partial s} Z_\D(s,P)\right|_{s=0}=\m_\D(P).\]
It is proven in \cite{Lalin-Roy} that 
\begin{equation}
\label{eq:LalRoy}
    Z_{\D}(s,x+1)=\frac{\Gamma(2+s)}{\Gamma\left(2+\frac{s}{2}\right)^2}.
\end{equation}

\section{The areal Mahler measure of $x+y+k$}\label{sec:x+y+k}

In this section, we prove Theorem \ref{thm:x+y+k}.  We remark that Lemma \ref{lem:Pritsker} gives the case $k\geq 2$ immediately, proving equation \eqref{eq:k>2}.
Our strategy for proving equation \eqref{eq:k<2} in  Theorem \ref{thm:x+y+k} is as follows. We will compute the areal Zeta Mahler function of $x+y+k$ in Subsection \ref{sec:azmmkxy}. For this, we will follow ideas of Borwein and Straub \cite{BS}, Borwein, Straub, Wan, and Zudilin \cite{BSWZ}, Ringeling \cite{Ringeling}, and others, to find $Z_{\D}(s, x+y+k)$ as a solution of a certain hypergeometric differential equation. We will then compute  $Z_{\D}'(0,x+y+k)$ to find $\m_\D(x+y+k)$ in Subsection \ref{sec:ammkxy}. The connection with $\m(x+y+k)$ is indirect, consisting of identifying one of the terms in the expression of 
$\m_\D(x+y+k)$ with the special value of a hypergeometric function known to be connected to  $\m(x+y+k)$ by a result of Berndt and Straub \cite{BerndtStraub}. This will be completed in Subsection \ref{sec:comparison}.

\subsection{The areal Zeta Mahler function of $x + y+k$}\label{sec:azmmkxy}
In this subsection we will prove the following result
\begin{thm}\label{thm:ZAMMk+x+y}
For $\re(s)>-2$ not an odd integer and $0\leq  k\leq 2$, we have 
\[Z_{\D}(s, x+y+k)=\alpha_0(s)\left(\frac{k}{2}\right)^{s+3}F_0\left(\frac{k^2}{4};s\right)+\alpha_1(s)F_1\left(\frac{k^2}{4};s\right),\]
where 
\begin{align}
 F_0(z;s) =& 
 \pFq{3}{2}{-\frac{1}{2},\frac{1}{2},\frac{3}{2}}{\frac{5+s}{2},\frac{5+s}{2}}{z}, \label{f0zs} 
 \\ F_1(z;s) =& \pFq{3}{2}{
 -2-\frac{s}{2},-1-\frac{s}{2}, -\frac{s}{2}}{1,-\frac{1+s}{2}}{z},\label{f1zs} 
\end{align}
\begin{align*}\alpha_0(s)  &= -\frac{2^{s+2} \tan \left(\frac{\pi  s}{2}\right) \Gamma\left(1+\frac{s}{2}\right)^2}{\pi  \Gamma\left(\frac{5+s}{2}\right)^2}
\end{align*}
and
\[\alpha_1(s) = \frac{4}{s + 4} \frac{\Gamma(2+s)}{\Gamma\left(2+\frac{s}{2}\right)^2}.\]
Moreover, for $k \geq 2$, 
\[
Z_{\D}(s, x + y + k) = k^s\cdot \pFq{3}{2}{-\frac{s}{2},-\frac{s}{2},\frac{3}{2}}{2,3}{\frac{4}{k^2}}.
\]
\end{thm}

This result implies that for all  $k \geq 0$, $Z_{\D}(s, x+y+k)$ extends to a meromorphic function on the complex plane. Specifically for $k > 2$ it is an entire function.
For $k < 2$, it seems that $Z_{\D}(s, x+y+k)$ has simple poles at all the even negative integers. Curiously enough, the pole structure for $k = 1$ is a bit different: it is actually analytic at $s = - 6$ \footnote{This follows from the fact that $F_0(z;-6) = \frac{1 - 4 z}{(1-z)^{5/2}}$, see \cite{wolfram-hypergeometricy}.}. This can also be seen in Figure \ref{IsThisCool?}.  When $k = 2$, there seems to be a simple pole at $s = - \frac{7}{2}$, however we do not know how to meromorphically continue the function to the entire complex plane in this case.  In Figure \ref{IsThisCool?} we can also see some of the real zeros of $Z_{\D}(s, x+y+k)$. Moreover, if we make a domain coloring plot of $Z_{\D}(s, x+y+1)$ (see Figure \ref{IsThisCool2?}), we see even more zeros. Interestingly enough it seems that these zeros are very close to the vertical line $\re(s) = - \frac{7}{2}$, this is also reflected in Table \ref{tab:zeros}. These observations can be compared with \cite[Theorem 1.3]{Ringeling}. There, it was shown that $Z(s, x + \frac{1}{x} + y + \frac{1}{y} + k)$ has all its non-trivial zeros exactly on the vertical line $\re(s) = - \frac{1}{2}$ for all real $k$.
\begin{figure}[h!]
        \centering
        \subfloat[$k = 1$]{%
            \centering\includegraphics[width=0.45\textwidth]{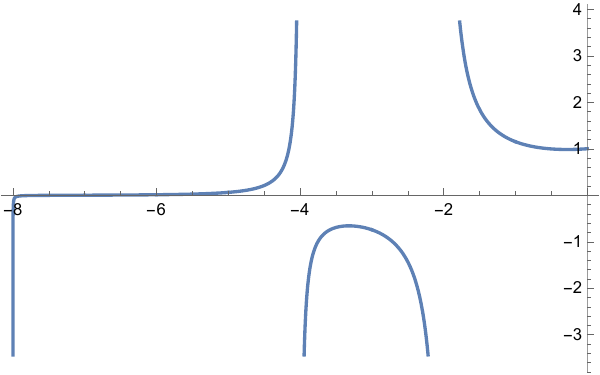}}
        \qquad
        \subfloat[$k = \frac{3}{2}$]{%
            \centering\includegraphics[width=0.45\textwidth]{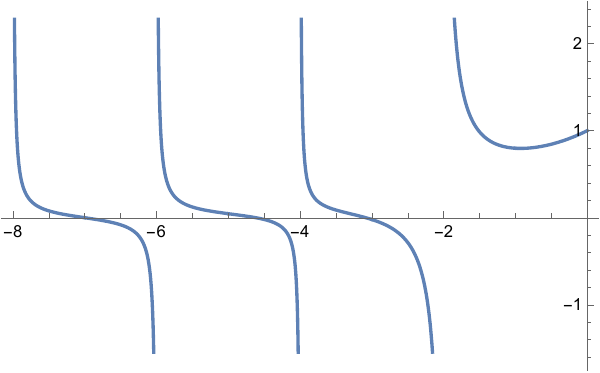}}
        \caption{\label{IsThisCool?} The graph of $Z_{\D}(s, k + x + y)$}
\end{figure}

\begin{figure}[h!]

            \centering\includegraphics[width=0.45\textwidth]{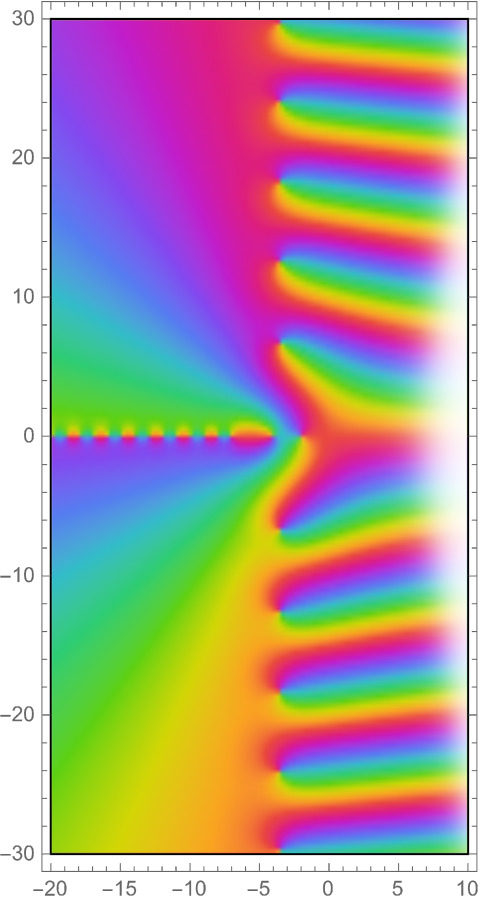}
        \caption{\label{IsThisCool2?} The graph of $Z_{\D}(s, 1 + x + y)$ for $-20 \leq \re(s) \leq 10$ and $|\im(s)| \leq 30$}
\end{figure}

\begin{table}[h!]
  \begin{center}
    \begin{tabular}{c} 
      \textbf{Zeros with increasing imaginary part}\\
      \hline
    $-3.4729  \ldots + \phantom{0}6.767  \ldots \cdot i$\\
    $-3.4918 \ldots + 12.656  \ldots \cdot i$\\
    $-3.4960 \ldots + 18.441  \ldots \cdot i$\\
    $-3.4977 \ldots + 24.194  \ldots \cdot i$\\
    $-3.4985 \ldots + 29.935  \ldots \cdot i$\\
    $-3.4992 \ldots + 41.398  \ldots \cdot i$\\
    $-3.4993 \ldots + 47.125  \ldots \cdot i$\\
    \end{tabular}
    \caption{non-real zeros of $Z_{\D}(s, x + y + 1)$}
    \label{tab:zeros}
  \end{center}
\end{table}


The argument for computing $Z_{\D}(s, x+y+k)$ for $k$ real and positive  starts with the case when $k > 2$. Recall that in this case, $\m_{\D}(x+y+k) = \log (k)$ from \eqref{eq:k>2}. 
We will first express the function $Z_{\D}(s, x+y+k)$ in terms of hypergeometric functions depending on $k$ and $s$ when $k>2$. We will then compute the differential equation satisfied by the expression obtained for $Z_{\D}(s, x+y+k)$. Finally, we will argue that, for general $k$, $Z_{\D}(s,  x + y+k)$ is, in fact, a particular solution of that differential equation.

 Let $X$ and $Y$ be random variables defined by walks of lengths $\rho_1$ and $\rho_2$ along the directions $\theta_1$ and $\theta_2,$ uniformly distributed on $[0, 1),$ respectively. In particular, $X$ takes values $x = \rho_1 e^{2\pi i\theta_1}$ and $Y$ takes values $y=\rho_2 e^{2\pi i \theta_2}.$ Let $Z$ be another random walk of unit length and direction $\theta.$ 
 
 We define a new random variable $T_1$ as the absolute value $|X + Y|.$ Let $p_{T_1}$ denote the probability density function of $T_1,$ which has support on $[0, 2).$ We further define the random variable $T_2$ as the absolute value $|X + Y+kZ|.$ Let $p_{T_2}$ denote the probability density function of $T_2$, which has support on $[k-2, k+2),$ when $k > 2$.

By definition \eqref{eq:ZAMMdefn1}, we have  \begin{align*}
     Z_{\D}(s, x+y+k ) =& \frac{1}{\pi^2}\int_{\D^2} |x+y+k|^s \dd A(x)\dd A(y).
\end{align*}
Parametrizing $x=\rho_1 e^{2i \pi \theta_1}$, $y=\rho_2e^{2i \pi \theta_2}$ with $0\leq\rho_1, \rho_2\leq 1$ and $0 \leq \theta_1, \theta_2\leq 1$, we obtain, 
\begin{align*}
   Z_{\D}(s, x+y+k )   =& 2^2 \int_{[0, 1]^4} | \rho_1 e^{2\pi i\theta_1} + \rho_2 e^{2 \pi i\theta_2}+k|^s \rho_1 \rho_2 \dd \rho_1 \dd \rho_2 \dd \theta_1 \dd \theta_2 \nonumber \\ =& 4 \int_{[0, 1]^5} | \rho_1 e^{2\pi i\theta_1} + \rho_2 e^{2 \pi i\theta_2}+ke^{2\pi i \theta}|^s \rho_1 \rho_2 \dd \rho_1 \dd \rho_2 \dd \theta_1 \dd \theta_2 \dd \theta. 
 \end{align*} 
 We now apply the change of variables $|x+y+ke^{2\pi i\theta}| = u$ and $|x+y| = v$ to obtain, for $k>2$, \begin{align}
      Z_{\D}(s, x+y+k) =& \frac{1}{\pi^2} \int_{0}^1\int_{\D}\int_{\D}|x + y+ k e^{2\pi i \theta} |^s \dd A(x)\dd A(y)\dd \theta \nonumber \\ =& \int_0^2\int_{k-v}^{k+v}u^s\,\PP(T_2=u\,|\,T_1=v)\,\PP(T_1=v) \dd u \dd v,\label{eq:exp}
 \end{align} where the normalized area measures of the variables $x$ and $y$ in the first integral are substituted by the new probability measures obtained from functions of the variables $x$ and $y,$ namely $u$ and $v.$ In \eqref{eq:exp}, we write \[\PP(T_2=u\,|\,T_1=v) = p_{T_2 \vert T_1}(u\vert v) \ \mbox{ and } \ \PP(T_1=v) = p_{T_1}(v),\] where \begin{equation}\label{condprob}
     \PP(T_2=u\,|\,T_1=v) = \frac{\PP(T_2 = u, T_1 = v)}{\PP(T_1 = v)} = \frac{p_{T_2, T_1}(u, v)}{p_{T_1}(v)} = p_{T_2 \vert T_1}(u|v).
 \end{equation} Here $p_{T_2, T_1}$ is the joint probability density function of the random variable $(T_2, T_1).$ 
 \begin{rem}\label{probability}
     Note that, for $k\geq 0$, since $T_2$ is dependent on $T_1,$ the moment generating function of $\log T_2$ can be expressed as follows: \begin{align}
         \E[e^{s\log T_2}] = \E[T_2^s]=& \int_{\max\{0,k-2\}}^{k+2} u^s \PP(T_2 = u) \dd u \nonumber \\ =& \int_{\max\{0,k-2\}}^{k+2} u^s \left(\int_{0}^2 \PP(T_2 = u, T_1 = v) \dd v\right) \dd u. \nonumber
         \end{align}
Since $T_1=|X+Y|$ and $T_2=|X+Y+kZ|$, we have that $\PP(T_2 = u, T_1 = v)=0$ unless $|k-v|<|u|<v+k$. We incorporate this restriction and exchange the order of integration to get 
         \begin{align}
       \E[e^{s\log T_2}]   =&  \int_{0}^2 \int_{|k-v|}^{k+v} u^s \PP(T_2 = u, T_1 = v) \dd u \dd v \nonumber \\ =& \int_{0}^2 \int_{|k-v|}^{k+v} u^s \PP(T_2 = u \vert T_1 = v)\PP(T_1 = v) \dd u \dd v, \label{conditionalt1t2}  
     \end{align} where the last equality follows from \eqref{condprob}. 
     After integrating with respect to $u,$ we obtain \begin{align*}
      \E[e^{s\log T_2}]   =& \int_{0}^2 \E[T_2^s \vert T_1 = v]\PP(T_1 = v) \dd v \nonumber \\ =& \E[\E[T_2^s \vert T_1]], 
     \end{align*}
     where the last equality follows from considering $\E[T_2^s \vert T_1 = v]$ as a function of $v$ (and therefore as a function of $T_1$), and then evaluating the expectation of the function with respect to the variable $T_1.$
     By combining \eqref{eq:exp} and \eqref{conditionalt1t2} together, we conclude that  the areal Zeta Mahler function of $x+y+k$ coincides with the moment generating function of $\log T_2,$ namely $\E\left[e^{s\log T_2}\right].$ 
 \end{rem}

It remains to compute \eqref{eq:exp}, for which we need some auxiliary results. The next statement will be applied to the computation of $ \PP(T_1=v)$ and is interesting in its own right. 
\begin{thm}\label{thm:areal-zeta-x+y} For $\re(s)>-2$, we have that 
\begin{equation*}
    Z_{\D} (s, x+y)=\frac{4}{s+4} \frac{\Gamma(2+s)}{\Gamma\left(2+\frac{s}{2}\right)^2}.
\end{equation*}
\end{thm}
\begin{proof}
By definition \eqref{eq:ZAMMdefn1}, we have 
\begin{align*}
Z_{\D} (s, x+y)=&\frac{1}{\pi^2}\int_\D \int_\D |x+y|^s \dd A(x)\dd A(y).
\end{align*}
Parametrizing $x=\rho_1 e^{2i \pi \theta_1}$, $y=\rho_2e^{2i \pi \theta_2}$ with $0\leq\rho_1, \rho_2, \theta_1, \theta_2 \leq 1$, we obtain, 
\begin{align*}
Z_{\D} (s, x+y)=&2\int_{[0,1]^4} |\rho_1e^{2i\pi \theta_1}+\rho_2e^{2i\pi \theta_2}|^s \rho_1 \rho_2 \dd \theta_1 \dd \theta_2 \dd \rho_1 \dd \rho_2\\=& 
2\int_{[0,1]^4}\rho_2^s |\rho_1\rho_2^{-1} e^{2i\pi (\theta_1-\theta_2)}+1|^s \rho_1 \rho_2 \dd \theta_1 \dd \theta_2 \dd \rho_1 \dd \rho_2\\
=&4 \int_{0\leq \rho_2\leq \rho_1\leq 1}\rho_1^{s+1}\rho_2\int_0^1  |\rho_2\rho_1^{-1} e^{2i\pi \tau}+1|^s \dd \tau \dd \rho_1 \dd \rho_2 \\ &+4 \int_{0\leq \rho_1\leq \rho_2\leq 1}\rho_1\rho_2^{s+1} \int_0^1  |\rho_1\rho_2^{-1} e^{2i\pi \tau}+1|^s \dd \tau \dd \rho_1 \dd \rho_2,
\end{align*}
where we have set $\tau=\theta_1-\theta_2$. 
We remark that the inner integrals correspond to the Zeta Mahler functions of $x+\rho_2\rho_1^{-1}$ and $x+\rho_1\rho_2^{-1}$.
Applying Theorem \ref{thm:Akatsuka}, we have that 
\begin{align*}
Z_{\D} (s, x+y)=& 4\int_{0\leq \rho_2\leq \rho_1\leq 1}\rho_1^{s+1}\rho_2\cdot 
\pFq{2}{1}{-\frac{s}{2},-\frac{s}{2}}{1}{\rho_2^2\rho_1^{-2} }
 \dd \rho_1 \dd \rho_2  \nonumber \\ &+4\int_{0\leq \rho_1\leq \rho_2\leq 1}\rho_1\rho_2^{s+1} \cdot
 \pFq{2}{1}{-\frac{s}{2},-\frac{s}{2}}{1}{\rho_1^2\rho_2^{-2} } \dd \rho_1 \dd \rho_2. \end{align*} Under the change of variables $\rho_1 \mapsto \rho_2$ and $\rho_2 \mapsto \rho_1,$ the second integral above transforms into the first integral. Further setting $\sigma=\rho_2\rho_1^{-1}$, we obtain
 \begin{align*}
    Z_{\D} (s, x+y)=& 8\int_{0\leq \rho_2\leq \rho_1\leq 1}\rho_1^{s+1}\rho_2 \cdot \pFq{2}{1}{-\frac{s}{2},-\frac{s}{2}}{1}{\rho_2^2\rho_1^{-2} }
 \dd \rho_1 \dd \rho_2\\
 =&8\int_0^1 \rho_1^{s+3}\int_0^1 \sigma \cdot
 \pFq{2}{1}{-\frac{s}{2},-\frac{s}{2}}{1}{\sigma^2 }
  \dd \sigma \dd \rho_1 \\ =& \frac{8}{s+4} \cdot \frac{1}{2} \int_{0}^1 
 \pFq{2}{1}{-\frac{s}{2},-\frac{s}{2}}{1}{\sigma^2 }
  \dd \sigma^2 \\ =& \frac{4}{s+4}
   \pFq{2}{1}{-\frac{s}{2},-\frac{s}{2}}{2}{1}\\
=&\frac{4}{s+4}\frac{\Gamma(2+s)}{\Gamma\left(2+\frac{s}{2}\right)^2},
\end{align*}
where the last equality follows from Theorem \ref{thm:Gauss} (by taking $a = b = -\frac{s}{2}$ and $c=2$).
\end{proof}

The following two lemmas give us formulas for  $\PP(T_1=v)$ and  $\PP(T_2=u\,|\,T_1=v)$.
\begin{lem}\label{densityt1}
    For $0 \leq v \leq 2,$ we have \begin{equation*}
        \PP(T_1=v) = p_{T_1}(v) =\frac{v}\pi\left(2\pi-v\sqrt{4-v^2}-4\arcsin\left(\frac{v}{2}\right)\right).
    \end{equation*}
\end{lem}

\begin{proof}
    Following the discussion in Remark \ref{probability}, we observe that the areal Zeta Mahler function of $x+y$ coincides with the moment generating function of $T_1,$ namely
    \[Z_{\D}(s, x+y) = \frac{4}{s+4} \frac{\Gamma(2+s)}{\Gamma\left(2+\frac{s}{2}\right)^2} =  \E [T_1^s] = \int_{0}^{\infty} v^s p_{T_1}(v) \dd v = \int_{0}^{2} v^s p_{T_1}(v) \dd v.\] 
        Further, the above equality implies that $Z_{\D}(s-1, x+y)$ is the Mellin transform of $p_{T_1}(v).$ We  can then retrieve $p_{T_1}(v)$ by considering the inverse Mellin transform of \[Z_{\D}(s-1, x+y) = \frac{4}{s+3}\frac{\Gamma(1+s)}{\Gamma\left(\frac{3+s}{2}\right)^2}.\] This can be computed by finding residues at the negative integers, and is given for $0\leq v \leq 2$ by 
        \[p_{T_1}(v) = 2v - \frac{4}{\pi}v^2 + 4\sum_{j=4}^{\infty}\frac{(-v)^{j}}{(j-3)\, \Gamma(j) \, \Gamma\left(\frac{3-j}{2}\right)^2} = \frac{v}\pi\left(2\pi-v\sqrt{4-v^2}-4\arcsin\left(\frac{v}{2}\right)\right).\]
    \end{proof}

\begin{lem}\label{densityt2t1}
    For $k\geq 0$,  $0 \leq v \leq 2$ and $|k-v|< u <  k+v,$ we have \begin{equation*}
      \PP(T_2=u\,|\,T_1=v) = p_{T_2\vert T_1}(u\vert v) = \frac{2u}{\pi\sqrt{4k^2v^2-(u^2-v^2-k^2)^2}}.
    \end{equation*}
\end{lem}

\begin{proof}
    We start by computing the cumulative distribution function of $T_2$ at $T_1 = v$ given by \[f_{v}(u) =  \PP(T_2 \leq u\,|\,T_1=v).\] We can then compute $p_{T_2\vert T_1}(u\vert v)$ as the  derivative of $f_{v}(u)$ with respect to $u$. 

    Writing $Z = e^{2\pi i \theta},$  where we now think of $\theta$ as uniformly distributed in $[-1/2, 1/2),$ we have \[f_{v}(u) = \PP(T_2 \leq u\,|\,T_1=v) = \PP(|ke^{2i\pi \theta} + v|\leq u\,|\,T_1=v) = \int_{[-1/2, 1/2) \cap I_{u, v}} \dd \theta,\] 
    where $I_{u, v}$ is the set of values of $\theta$ such that \[
T_2^2=|ke^{2i \pi \theta} + v|^2 = v^2+k^2+2vk\cos(2\pi \theta) \leq u^2 \Longleftrightarrow  \cos(2\pi \theta) \leq \frac{u^2-v^2-k^2}{2vk}.
\] Since $v \in [0, 2]$ and $u \in (|k-v|, k+v),$ the quantity  $\beta = \frac1{2\pi}\arccos\left(\frac{u^2-v^2-k^2}{2vk}\right) \in \left[0, \frac{1}{2}\right)$ is well-defined and 
$I_{u, v}$ is a non-empty set.  Then \[f_{v}(u) = \int_{[-1/2, 1/2) \cap I_{u, v}} \dd \theta = 1 - \int_{-\beta}^{\beta} \dd \theta = 1 - 2\beta = 1 - \frac{1}{\pi}\arccos\left(\frac{u^2-v^2-k^2}{2vk}\right).\] Taking the derivative with respect to $u,$ we obtain the required expression.
\end{proof}

 We remark that Lemmas \ref{densityt1} and \ref{densityt2t1} are derived independently of the assumption $k> 2$ or $0\leq k \leq 2.$

We are now ready to compute the integral in \eqref{eq:exp}. When $k > 2,$ evaluation of the integral is a consequence of the following two crucial lemmas.

\begin{lem}
\label{lem:Lem1}
    For $\re(s) > 0$, $k > 2$, and $0\leq v \leq 2$, we have
\begin{equation*}
        \int_{k-v}^{k+v} u^s \cdot  \PP(T_2=u\,|\,T_1=v) \dd u = k^s \cdot 
 \pFq{2}{1}{-\frac{s}{2},-\frac{s}{2}}{1}{\frac{v^2}{k^2}}.
    \end{equation*}
\end{lem}
\begin{proof}
Since $k > 2,$ we have $k > v.$ By Lemma \ref{densityt2t1}, we have
\[ \int_{k-v}^{k+v} u^s \cdot  \PP(T_2=u\,|\,T_1=v) \dd u= \frac{2}{\pi}\int_{k-v}^{k+v} u^s \frac{u}{\sqrt{4k^2v^2-(u^2-v^2-k^2)^2}} \dd u.\]
Setting $w=u^2$, the above integral equals
\begin{equation}\label{tag}
    \frac{1}{\pi}\int_{(k-v)^2}^{(k+v)^2}\frac{w^{\frac{s}{2}}}{\sqrt{4k^2v^2-(w-v^2-k^2)^2}} \dd w.
\end{equation}
After the change of variables
\[
w'=\frac{w-(k-v)^2}{(k+v)^2-(k-v)^2}=\frac{w-(k-v)^2}{4kv},
\]
\eqref{tag} becomes 
\[
\frac{1}{\pi}\int_0^1\frac{(4kvw' +(k-v)^2)^{\frac{s}{2}}}{\sqrt{w'(1-w')}} \dd w'=(k-v)^s\cdot
 \pFq{2}{1}{-\frac{s}{2},\frac{1}{2}}{1}{\frac{-4kv}{(k-v)^2}}.
\]
The last equality and the convergence for $\re(s)>0$ follow from the integral representation of the hypergeometric function given in Theorem \ref{thm:hyperint}. 

Applying the quadratic transformation for the hypergeometric function described in \cite{wolfram-hyperquad}, we find, for $k > 2 \geq v,$
\[(k-v)^s\cdot
 \pFq{2}{1}{-\frac{s}{2},\frac{1}{2}}{1}{\frac{-4kv}{(k-v)^2}}= k^s\cdot \pFq{2}{1}{-\frac{s}{2},-\frac{s}{2}}{1}{\frac{v^2}{k^2}}.\]
This concludes the proof.
\end{proof}

We need the next lemma to completely evaluate the integral in \eqref{eq:exp} for $k > 2.$
\begin{lem}
\label{Lem3}
For $\re(s) > 0$ and $k>2$, we have 
\[\int_0^2 p_{T_1}(v) \cdot 
 \pFq{2}{1}{-\frac{s}{2},-\frac{s}{2}}{1}{\frac{v^2}{k^2}} \dd v = \pFq{3}{2}{-\frac{s}{2},-\frac{s}{2}, \frac{3}{2}}{2,3}{\frac{4}{k^2}}.\]
\end{lem}
\begin{proof}
    We expand the hypergeometric function into its series and interchange the integral and the sum. In other words, we have
    \begin{align*}
        \int_0^2 p_{T_1}(v) \cdot 
        \pFq{2}{1}{-\frac{s}{2},-\frac{s}{2}}{1}{\frac{v^2}{k^2}} \dd v  =&
    \sum_{n = 0}^\infty \frac{\left( -\frac{s}{2}\right)^2_n}{n!^2} \frac{1}{k^{2n}}\int_0^2 v^{2n} p_{T_1}(v) \dd v \\ =& \sum_{n = 0}^\infty \frac{\left( -\frac{s}{2}\right)^2_n}{n!^2} \frac{1}{k^{2n}} \frac{4^n\left( \frac{3}{2}\right)_n}{(n+1) \left( 3 \right)_n},
    \end{align*} where the last equality follows from  the discussion in Remark \ref{probability} and Theorem \ref{thm:areal-zeta-x+y}, since \[\int_0^2 v^{2n} p_{T_1}(v) \dd v = Z_{\D}(2n, x+y) = \frac{4}{2n+4}\frac{\Gamma(2n+2)}{\Gamma(n+2)^2}=  \frac{4^n\left( \frac{3}{2}\right)_n}{(n+1) \left( 3 \right)_n}.\]
Furthermore, \[\sum_{n = 0}^\infty \frac{\left( -\frac{s}{2}\right)^2_n}{n!^2} \frac{1}{k^{2n}} \frac{4^n\left( \frac{3}{2}\right)_n}{(n+1) \left( 3 \right)_n} = \sum_{n=0}^{\infty}\frac{\left( -\frac{s}{2}\right)^2_n \left( \frac{3}{2}\right)_n}{(3)_n (2)_n}\frac{\left(\frac{4}{k^2}\right)^n}{n!}\] coincides with the series representation of $
\pFq{3}{2}{-\frac{s}{2},-\frac{s}{2},\frac{3}{2}}{2,3}{\frac{4}{k^2}},$ and this concludes the proof.
\end{proof}

Combining Lemmas \ref{lem:Lem1} and \ref{Lem3}, we have, for $\re(s) > 0$ and $k > 2$,  \begin{equation}\label{k>2}
    Z_{\D}(s, x+y+k) = k^s\cdot \pFq{3}{2}{-\frac{s}{2},-\frac{s}{2},\frac{3}{2}}{2,3}{\frac{4}{k^2}}.
\end{equation}

We now extend the above result for $k \leq 2.$ Notice that in this case the boundary points of the inner integral in \eqref{eq:exp} will be $|k-v|$ and $k+v,$ since for $0 \leq v \leq 2,$ there exists at least one $v$ such that $k \leq v.$ 
Denote \begin{equation}\label{F(k)}
   \mathcal{F}(k):=  k^s \cdot 
    \pFq{2}{1}{-\frac{s}{2},-\frac{s}{2}}{1}{\frac{v^2}{k^2}}.
\end{equation} Now, to evaluate the integral in \eqref{eq:exp} with boundary points $|k-v|$ and $k+v,$ we have the following lemma.

\begin{lem}
\label{Lem2}
    Let $s$ be a real positive number that is not an odd integer. Then, for $k \leq 2,$ we have
\begin{equation*}
        \int_{|k-v|}^{k+v}  u^s \cdot \PP(T_2=u\,|\,T_1=v)  \dd u = \re \left(\mathcal{F}(k)\right) - \cot\left( \frac{\pi s}{2}\right) \im \left(\mathcal F(k)\right).
    \end{equation*}
Here $\mathcal{F}(k) := \lim_{\varepsilon \to 0^{+}} \mathcal{F}(k + i \varepsilon),$ where we use the analytic continuation of $\mathcal{F}(z)$ to the complex upper half-plane. 

\end{lem}
\begin{proof}
    For $k \geq v$, the result follows as in  Lemma \ref{lem:Lem1}, since the integral is real, i.e. $\im(\mathcal{F}(k)) = 0.$ Assume that $k < v$.  We first analyze the integral 
    \[
    \int_{k-v}^{k+v} u^s \frac{u}{\sqrt{4k^2v^2-(u^2-v^2-k^2)^2}} \dd u,
    \]
    where $\sqrt{\cdot}$ indicates the principal branch of the square root and $u^s = e^{s \mathrm{Log}(u)}$, where $\mathrm{Log}$ is the principal branch of the logarithm.
    In the range $k - v \leq u \leq k + v$, the expression inside the square root $4 k^2 v^2 - (u^2 - v^2 - k^2)^2 = -(k - u - v) (k + u - v) (k - u + v) (k + u + v)$ becomes negative whenever $k - v < u < v - k$. By the choice of branch cut for the square root, it follows that 
    \[
    \lim_{\varepsilon \to 0^{+}} \int_{k + i \varepsilon-v}^{k + i \varepsilon+v} u^s \frac{u}{\sqrt{4(k + i \varepsilon)^2v^2-(u^2-v^2-(k + i \varepsilon)^2)^2}} \dd u = \int_{k - v}^{k +v} u^s \frac{u}{\sqrt{4k^2v^2-(u^2-v^2-k^2)^2}} \dd u.
    \]
    By the analytic continuation of $\mathcal{F}(z)$   to the complex upper half-plane given in \eqref{F(k)}, we can write
    \[
    \mathcal{F}(k) := \lim_{\varepsilon \to 0^{+}} \mathcal{F}(k + i \varepsilon) = \frac{2}{\pi} \int_{k - v}^{k +v} u^s \frac{u}{\sqrt{4k^2v^2-(u^2-v^2-k^2)^2}} \dd u.
    \]
Thus, we have,
    \begin{align}
        \mathcal{F}(k) =&\frac{2}{\pi}\int_{k-v}^{k+v} u^s \cdot\frac{ u \dd u}{\sqrt{4k^2v^2-(v^2-u^2-k^2)^2}}  \nonumber \\ =&\frac{2}{\pi} \int_{v - k}^{k+v}   u^s \cdot\frac{ u \dd u}{\sqrt{4k^2v^2-(v^2-u^2-k^2)^2}}  + \frac{2}{\pi} \int_{0}^{v - k}   u^s \cdot\frac{ u \dd u}{\sqrt{4k^2v^2-(v^2-u^2-k^2)^2}}  \nonumber \\ &+\frac{2}{\pi}  \int_{k - v}^{0}  u^s \cdot\frac{ u \dd u}{\sqrt{4k^2v^2-(v^2-u^2-k^2)^2}}  \label{covneeded} \\ =& \frac{2}{\pi} \int_{v-k}^{k+v}   \frac{ u^{s+1} \dd u}{\sqrt{4k^2v^2-(v^2-u^2-k^2)^2}}  +\frac{2}{\pi}  (1 - e^{\pi i s}) \int_0^{v - k}   \frac{ u^{s+1} \dd u}{\sqrt{4k^2v^2-(v^2-u^2-k^2)^2}} \label{intforK<1}.
    \end{align}
The last equality follows from applying the change of variables $u \to -u$ to the third integral in \eqref{covneeded}. Observe that in \eqref{covneeded}, the integrand is real on the interval $[v-k,k+v],$ while it is purely imaginary on the intervals $[0, v-k]$ and $[k-v, 0].$ Note that we need $s$ to be real to express $(-1)^s$ as $e^{\pi i s}$ in \eqref{intforK<1}. Next, taking the complex conjugate of \eqref{intforK<1}, we have
\[\overline{\mathcal{F}(k)} = \frac{2}{\pi} \int_{v-k}^{k+v} \frac{ u^{s+1} \dd u}{\sqrt{4k^2v^2-(v^2-u^2-k^2)^2}} - \frac{2}{\pi}  (1 - e^{-\pi i s})\int_0^{v-k} \frac{u^{s+1} \dd u}{\sqrt{4k^2v^2-(v^2-u^2-k^2)^2}}.\]
Combining the above equality with \eqref{intforK<1}, we finally have
\begin{align*}
   \frac{2}{\pi}  \int_{v - k}^{k+v} \frac{ u^{s+1}}{\sqrt{4k^2v^2-(v^2-u^2-k^2)^2}} \dd u =& \frac{\mathcal{F}(k) - e^{\pi i s} \overline{\mathcal{F}(k)}}{1 - e^{\pi i s}} \\ =& \re \left(\mathcal{F}(k)\right) - \cot\left( \frac{\pi s}{2}\right) \im \left(\mathcal{F}(k)\right).
\end{align*} \end{proof}
Now integrating $p_{T_1}(v)\cdot  \mathcal{F}(k)$ as in Lemma \ref{Lem3} and considering the real and imaginary parts, we reach the following result.
\begin{prop}
\label{Thm1}
    For real $s > 0$, not an odd integer, and $k>0$,
    \begin{equation}\label{k<2}
        Z_{\D}(s,x+y+k) = \re \left(\mathcal{G}(k)\right) - \cot \left( \frac{\pi s}{2}\right) \im \left(\mathcal{G}(k)\right),
    \end{equation}
    where, for $k>2$,  \[\mathcal{G}(k) := k^s \cdot 
    \pFq{3}{2}{-\frac{s}{2},-\frac{s}{2},\frac{3}{2}}{2,3}{\frac{4}{k^2}}\]
 and for $0<k\leq 2$,  \[\mathcal{G}(k) := \lim_{\varepsilon \to 0^{+}} \mathcal{G}(k + i \varepsilon),\] where we use the analytic continuation of $\mathcal{G}(z)$ to the complex upper half-plane. 
\end{prop}

A consequence of Proposition \ref{Thm1} is that $Z_{\D}(s,  x + y+k)$ satisfies the same differential equation for both $k\leq 2$ and $k>2,$ as both involve the same hypergeometric function $\mathcal{G}(k)$ in their expressions (compare equation \eqref{k>2} with equation \eqref{k<2}). For $z = \frac{k^2}{4},$ define \begin{equation*}
    \tilde{\mathcal{G}}(z) := \mathcal{G}(\sqrt{4z}) = (4z)^{\frac{s}{2}}\cdot 
    \pFq{3}{2}{-\frac{s}{2},-\frac{s}{2},\frac{3}{2}}{2,3}{\frac{1}{z}}.
\end{equation*}
Our goal is to find a differential equation satisfied by $\tilde{\mathcal{G}}(z)$ and a particular solution $H(z) = H\left(\frac{k^2}{4}\right)$ to the differential equation which coincides with $\mathcal{G}(k)$ when $k > 2$ and $s > 0.$ Then, by the analytic properties of $H\left(\frac{k^2}{4}\right),$ we will conclude that $Z_{\D}(s, x+y+k) = H\left(\frac{k^2}{4}\right)$ for all $k \in \C$ and for a larger region of $s \in \C.$

 In order to obtain a differential equation satisfied by $\tilde{\mathcal{G}},$ we start with a hypergeometric differential equation satisfied by $\pFq{3}{2}{-\frac{s}{2},-\frac{s}{2},\frac{3}{2}}{2,3}{z}.$ Let $\theta := z\frac{\dd}{\dd z}.$ Then, replacing $p =3, q=2, a_1 = a_2 = -\frac{s}{2}, a_3 = \frac{3}{2}, b_1 = 2,$ and $b_3 = 3$ in \eqref{diffeq}, we obtain the third order differential equation \begin{align*}
    \left[z\left(\theta - \frac{s}{2}\right)\left(\theta - \frac{s}{2}\right)\left(\theta + \frac{3}{2}\right) - \theta\left(\theta + 1\right)\left(\theta + 2\right)\right]\mathcal{W}(z) =& 0 \\ \Longleftrightarrow \left[\left(z-1\right)\theta^3-\left(3 - \frac{3z}{2} + zs\right)\theta^2-\left(2 + \frac{3zs}{2} - \frac{zs^2}{4}\right)\theta+\frac{3zs^2}{8}\right] \mathcal{W}(z) =& 0,
\end{align*}
satisfied by \[4^{-\frac{s}{2}} \cdot z^{\frac {s}{2}}\tilde{\mathcal{G}}(z^{-1})=
\pFq{3}{2}{-\frac{s}{2},-\frac{s}{2},\frac{3}{2}}{2,3}{z}.\] Now, to find the differential equation satisfied by $\tilde{\mathcal{G}},$ we substitute $\mathcal{W}(z)$ with $4^{-\frac{s}{2}} \cdot z^{\frac {s}{2}}\tilde{\mathcal{G}}(z^{-1})$ above. Next, using the change of variables $z \mapsto \frac{1}{z}$ and further simplifying, we obtain a third order differential equation using \eqref{diffeq} \begin{equation}\label{reqdiffeq}
    s (8+6 s+s^2) \mathcal{V}(z) -2 (2+3 s^2 z+s (2+6 z)) \mathcal{V}' (z) -4 z (-3+s-3 s z)\mathcal{V}''(z)-8(z-1) z^2 \mathcal{V}'''(z) = 0,
\end{equation}satisfied by $\tilde{\mathcal{G}}(z).$  Dividing both sides of \eqref{reqdiffeq} by the coefficient of $\mathcal{V}'''(z),$ we find that this differential equation has a regular singularity at $z = 0$ (see \cite{Beukers}).

It remains to obtain a fundamental set of solutions of \eqref{reqdiffeq} around $z =0.$ Using the method of Frobenius to find power-series solutions to differential equations \cite{Norlund}\footnote{See \cite[§16.3]{Ince} for more details.}, we show that the local exponents at $z=0$ are $0, 0,$ and $\frac{3+s}{2},$ and we further obtain that the differential equation has a basis of solutions around $z = 0$ of the form 
\[ z^{\frac{3+s}{2}} F_0(z;s), \qquad F_1(z;s), \qquad \text{and} \qquad F_2(z;s) + \log(z)F_1(z;s).\]
Here $F_0, F_1$, and $F_2$ are holomorphic and non-zero at $z = 0$, and are respectively given by equations \eqref{f0zs}, \eqref{f1zs}, and  
\begin{align*}F_2(z;s) =& 
 G_{4,3}^{\,3,3}\!\left(\left.{\begin{matrix}
 1+\frac{s}{2}, 2+\frac{s}{2}, 3+\frac{s}{2}\\[5pt]0, 0, \frac{3+s}{2}  \end{matrix}}\;\right|\,z\right), 
\end{align*}
where $ G_{4,3}^{\,3,3}$ is a Meijer $G$-functions given by $\eqref{eq:G}$.

Now we have all the elements to prove Theorem \ref{thm:ZAMMk+x+y}.

\begin{proof}[Proof of Theorem \ref{thm:ZAMMk+x+y}]
    From Theorem \ref{thm:areal-zeta-x+y}, we have, for almost all\footnote{The only problematic $s$ are coming from the zeros of $\Gamma\left(2+\frac{s}{2}\right)$ and $s+4,$ and the poles of $\Gamma(2+s).$} $s,$ that $Z_{\D}(s, x+y+k)$ converges as $k \to 0$. For $z = \frac{k^2}{4},$ this eliminates $F_2(z;s) + \log(z)F_1(z;s)$ as a possible contributor to the expression of $Z_{\D}(s, x+y+k),$ since it does not converge as $k\rightarrow 0$. Therefore, $Z_{\D}(s, x+y+k)$ is a linear combination of $z^{\frac{3+s}{2}} F_0(z; s)$ and $F_1(z; s).$ 

Further, using the initial conditions $Z_{\D}(s, x+y+2) = \mathcal{G}(2)$ and $Z_{\D}(s, x+y) =  \frac{4}{s + 4} \frac{\Gamma(2+s)}{\Gamma\left(2+\frac{s}{2}\right)^2} = \mathcal{G}(0)$ we conclude that
\begin{equation}\label{AZMFchapt}
Z_{\D}(s, x+y+k) = \alpha_0(s) \left(\frac{k}{2}\right)^{s+3} F_0 \left( \frac{k^2}{4};s\right) + \alpha_1(s) F_1 \left( \frac{k^2}{4};s\right),\end{equation}
where $\alpha_1(s) = \frac{4}{s + 4} \frac{\Gamma(2+s)}{\Gamma\left(2+\frac{s}{2}\right)^2}$ and
\begin{align*}\alpha_0(s) &= \frac{\mathcal{G}(2) - \alpha_1(s)F_1(1;s)}{F_0(1;s)}\\ &= \frac{ 2^s \cdot
\pFq{3}{2}{-\frac{s}{2},-\frac{s}{2},\frac{3}{2}}{2,3}{1} - \frac{4}{s + 4} \frac{\Gamma(2+s)}{\Gamma\left(2+\frac{s}{2}\right)^2} \cdot 
\pFq{3}{2}{-2-\frac{s}{2},-1-\frac{s}{2},-\frac{s}{2}}{1,-\frac{1+s}{2}}{1}}{
\pFq{3}{2}{-\frac{1}{2},\frac{1}{2},\frac{3}{2}}{\frac{5+s}{2},\frac{5+s}{2}}{1}
}.\end{align*}
Since the hypergeometric series $_{q+1}F_{q}\!\left(\left.{\begin{matrix}a_{1},\dots ,a_{q+1}\\b_{1},\dots ,b_{q}\end{matrix}}\;\right|\,z\right)$ absolutely converges at $z = 1$ when \\$\re\left(\sum_{j=1}^q b_j - \sum_{\ell=1}^{q+1} a_{\ell}\right) > 0,$ we have that both $F_0(z; s)$ and $F_1(z; s)$ converge absolutely for $\re(s) > -\frac{7}{2}.$ Since $\Gamma(2+s)$ has no poles at $\re(s)>-2$ and a simple pole at $\re(s)=-2$, the expression in \eqref{AZMFchapt} converges absolutely in $\re(s) > -2,$ and this concludes the proof.

If we apply the formula in \cite{wolfram-hypergeometricx} with $a_1 = a_2 = -\frac{s}{2}$, $a_3 = \frac{3}{2}$, $b_1 = 2$ and $b_2 = 3$, this gives the identity:
\begin{align}
   \pFq{3}{2}{-\frac{s}{2},-\frac{s}{2},\frac{3}{2}}{2,3}{1} &= \frac{2 \Gamma \left( 1 + \frac{s}{2}\right)\Gamma \left( \frac{3+s}{2}\right)}{\Gamma \left( \frac{3}{2}\right) \Gamma \left( 2 + \frac{s}{2}\right) \Gamma \left(3 +  \frac{s}{2}\right) } \cdot  \pFq{3}{2}{-2-\frac{s}{2},-1-\frac{s}{2},-\frac{s}{2}}{1,-\frac{1+s}{2}}{1} \nonumber \\
   & \quad  + \frac{2 \Gamma \left( 1 + \frac{s}{2}\right)\Gamma \left(- \frac{3+s}{2}\right) }{\Gamma \left( \frac{1}{2}\right) \Gamma \left( \frac{3}{2}\right) \Gamma \left( -\frac{s}{2}\right) \Gamma \left( \frac{5+s}{2}\right)} \cdot \pFq{3}{2}{-\frac{1}{2},\frac{1}{2},\frac{3}{2}}{\frac{5+s}{2},\frac{5+s}{2}}{1} \label{eq:3f2-identity}
\end{align}
valid for $\re(s) > -2$. By the duplication formula for the Gamma function we find
\[
2^s \frac{2 \Gamma \left( 1 + \frac{s}{2}\right)\Gamma \left( \frac{3+s}{2}\right)}{\Gamma \left( \frac{3}{2}\right) \Gamma \left( 2 + \frac{s}{2}\right) \Gamma \left(3 +  \frac{s}{2}\right) } = \frac{4}{s + 4} \frac{\Gamma(2+s)}{\Gamma\left(2+\frac{s}{2}\right)^2}.
\]
Using \eqref{eq:3f2-identity}, this implies that
\begin{align*}
  \alpha_0(s)  &= \frac{ 2^s \cdot
\pFq{3}{2}{-\frac{s}{2},-\frac{s}{2},\frac{3}{2}}{2,3}{1} - \frac{4}{s + 4} \frac{\Gamma(2+s)}{\Gamma\left(2+\frac{s}{2}\right)^2} \cdot 
\pFq{3}{2}{-2-\frac{s}{2},-1-\frac{s}{2},-\frac{s}{2}}{1,-\frac{1+s}{2}}{1}}{
\pFq{3}{2}{-\frac{1}{2},\frac{1}{2},\frac{3}{2}}{\frac{5+s}{2},\frac{5+s}{2}}{1}
} \\
&= \frac{2^s \frac{2 \Gamma \left( 1 + \frac{s}{2}\right)\Gamma \left(- \frac{3+s}{2}\right) }{\Gamma \left( \frac{1}{2}\right) \Gamma \left( \frac{3}{2}\right) \Gamma \left( -\frac{s}{2}\right) \Gamma \left( \frac{5+s}{2}\right)} \cdot \pFq{3}{2}{-\frac{1}{2},\frac{1}{2},\frac{3}{2}}{\frac{5+s}{2},\frac{5+s}{2}}{1}}{\pFq{3}{2}{-\frac{1}{2},\frac{1}{2},\frac{3}{2}}{\frac{5+s}{2},\frac{5+s}{2}}{1}} = 2^s \frac{2 \Gamma \left( 1 + \frac{s}{2}\right)\Gamma \left(- \frac{3+s}{2}\right) }{\Gamma \left( \frac{1}{2}\right) \Gamma \left( \frac{3}{2}\right) \Gamma \left( -\frac{s}{2}\right) \Gamma \left( \frac{5+s}{2}\right)}.
\end{align*}
The latter expression can be written in terms of the tangent function using the reflection formula for the Gamma function:
\[
\alpha_0(s) = 2^s \frac{2 \Gamma \left( 1 + \frac{s}{2}\right)\Gamma \left(- \frac{3+s}{2}\right) }{\Gamma \left( \frac{1}{2}\right) \Gamma \left( \frac{3}{2}\right) \Gamma \left( -\frac{s}{2}\right) \Gamma \left( \frac{5+s}{2}\right)} = -\frac{2^{s+2} \tan \left(\frac{\pi  s}{2}\right) \Gamma\left(1+\frac{s}{2}\right)^2}{\pi  \Gamma\left(\frac{5+s}{2}\right)^2}.
\]
\end{proof}

\subsection{The areal Mahler measure of $x + y+k$}\label{sec:ammkxy}
In this subsection we prove the following result
\begin{prop} For $0\leq k \leq 2$, we have
\begin{equation}\label{mahlerarealkxy}
    \m_{\D}(x + y + k)= -\frac{4 k^3}{9 \pi}\cdot \pFq{3}{2}{-\frac{1}{2},\frac{1}{2},\frac{3}{2}}{\frac{5}{2},\frac{5}{2}}{\frac{k^2}{4}}+ \frac{k^2}{2} - \frac{1}{4}.
\end{equation}
\end{prop}

\begin{proof}
The areal Mahler measure of $x+y+k$ can be obtained by computing $\frac{\dd}{\dd s} Z_{\D}(s, x+y+k) |_{s=0}$.
For $k \geq 2$, we find that $\m_{\D}(x+y+k) = \log (k)$.
For $k < 2$, we differentiate the expression \eqref{AZMFchapt} with respect to $s$.
By expanding the hypergeometric series in the numerator of $\alpha_0(s),$ we find that
\[\alpha_0(s)  = -\frac{32}{9 \pi} s + \mathcal{O}(s^2).\]
Moreover, we also have that
\begin{align*}\alpha_1(s) &= 1 - \frac{1}{4}s + \mathcal{O}(s^2),\\
F_0\left(\frac{k^2}{4};s\right)&= 
\pFq{3}{2}{-\frac{1}{2},\frac{1}{2},\frac{3}{2}}{\frac{5}{2},\frac{5}{2}}{\frac{k^2}{4}}+ \mathcal{O}(s),\\
F_1\left(\frac{k^2}{4};s\right)&= 1 + \frac{k^2}{2}s + \mathcal{O}(s^2). \end{align*}
Therefore, computing $\frac{\dd}{\dd s} Z_{\D}(s, x+y+k) |_{s=0},$ we conclude that
\begin{equation*}
    \m_{\D}(x + y + k)= -\frac{4 k^3}{9 \pi}\cdot \pFq{3}{2}{-\frac{1}{2},\frac{1}{2},\frac{3}{2}}{\frac{5}{2},\frac{5}{2}}{\frac{k^2}{4}}+ \frac{k^2}{2} - \frac{1}{4},
\end{equation*} which proves the desired result.
\end{proof}

\subsection{The  Mahler measure of $x + y+k$ and the Proof of Theorem \ref{thm:x+y+k}}\label{sec:comparison}
By Remark \ref{rem:CM}, the Mahler measure of $x + y + k$ for $0 < k < 2$ is given by
\begin{equation}\label{CMZAMM}
     \m(x + y + k) = \frac{1}{\pi} \left( 2 \log (k) \arcsin\left(\frac{k}{2}\right) + D \left( e^{2 i \arcsin\left(\frac{k}{2}\right)}\right)\right),
\end{equation} where $\arcsin\left(\frac{k}{2}\right) \in [0, \pi).$ Since \eqref{mahlerarealkxy} is written in terms of hypergeometric series in $k,$ we need a hypergeometric formula for $\m(x+y+k).$ The following result provides such a formula.
\begin{prop}
   For $0\leq k\leq 2,$ we have \begin{equation}\label{eq:hypgeok+x+y}
       \m(x +y+k) =  \frac{k}{\pi}\cdot 
       \pFq{3}{2}{\frac{1}{2},\frac{1}{2},\frac{1}{2}}{\frac{3}{2},\frac{3}{2}}{\frac{k^2}{4}}.
   \end{equation} 
\end{prop}
We remark that the argument to  derive \eqref{eq:hypgeok+x+y} using the Zeta Mahler function is analogous to the ones in Sections \ref{sec:azmmkxy} and \ref{sec:ammkxy}. Indeed, we can  prove the following result. 
\begin{thm} \label{thm:ZMMk+x+y} For $\re(s)>-1$ and $k \geq 0$, we have 
\[Z(s,x+y+k)=\beta_0(s)\left(\frac{k}{2}\right)^{1+s} G_0\left(\frac{k^2}{4};s\right)+\beta_1(s) G_1\left(\frac{k^2}{4};s\right),\]
where 
\begin{equation}\label{eq:Gs}
G_0(z;s)=\pFq{3}{2}{\frac{1}{2},\frac{1}{2}, \frac{1}{2}}{\frac{3+s}{2},\frac{3+s}{2}}{z}, \qquad G_1(z;s)=\pFq{3}{2}{-\frac{s}{2},-\frac{s}{2}, -\frac{s}{2}}{1,\frac{1-s}{2}}{z},\end{equation}
\[\beta_0(s)=\frac{1}{2^{s}} \tan\left(\frac{\pi s}{2}\right) 
\left(\frac{\Gamma(1+s)}{\Gamma\left(\frac {1+s}{2}\right)\Gamma\left(\frac{3+s}{2}\right)}\right)^2
\qquad
\text{and}\qquad \beta_1(s)=\frac{\Gamma(1+s)}{\Gamma\left(1+\frac{s}{2}\right)^2}.\]
\end{thm}
We consider the random variables  $S_1=|X+Y|$ and $S_2=|X+Y+k|$. Then for $k>2$, we have 
\begin{align}\label{eq:Zk}
Z(s,x+y+k)= \int_0^2 \int_{k-v}^{k+v}u^s \cdot\PP(S_2=u|S_1=v)\PP(S_1=v)\dd u \dd v.   
\end{align}
\begin{lem}
For $0\leq v \leq 2$,  
\[\PP(S_1=v) = \frac{2}{\pi \sqrt{4-v^2}}.\] 
\end{lem}
\begin{proof}
By equation \eqref{eq:KLO}
\[Z(s,x+y) = Z(s,x+1) = \frac{\Gamma(1+s)}{\Gamma\left(1+\frac{s}{2}\right)^2} =\frac{2^s}{\sqrt{\pi}}\frac{\Gamma\left(\frac{1+s}{2}\right)}{\Gamma\left(1+\frac{s}{2}\right)}  = \int_{0}^{\infty} v^s \PP(S_1=v) \dd v.\] Appying  the Mellin inverse to $Z(x+y; s-1)$, we get the result.  
\end{proof}

Recall by Lemma \ref{densityt2t1} that for  $k\geq 0$,  $0 \leq v \leq 2$ and $|k-v|< u <  k+v,$ we have 
\[\PP(S_2=u\,|\,S_1=v) = \frac{2u}{\pi \sqrt{4k^2 v^2 - (u^2 - v^2 - k^2)^2}}.\]
Then, from Lemmas \ref{lem:Lem1} and \ref{Lem2} we have, for $s$ real 
\begin{align}\label{eq:intS2S1}
\int_{|k-v|}^{k+v}  u^s \cdot \PP(S_2=u\,|\,S_1=v)  \dd u =& \int_{|k-v|}^{k+v} u^s \cdot
\frac{2u}{\pi \sqrt{4k^2 v^2 - (u^2 - v^2 - k^2)^2}}
 \dd u \nonumber \\=& \re (\mathcal{F}(k)) - \cot\left( \frac{\pi s}{2}\right) \im (\mathcal{F}(k)),
 \end{align}where $\mathcal{F}(k)$ is given by \eqref{F(k)}.

\begin{prop} For real $s>0$, not an odd integer, and $k>0$, 
\[Z(s,x+y+k)=\re (\mathcal{J}(k)) - \cot\left( \frac{\pi s}{2}\right) \im (\mathcal{J}(k)),\]
where, for $k>2$, 
\[\mathcal{J}(k):=k^s \cdot
  \pFq{3}{2}{-\frac{s}{2},-\frac{s}{2}, \frac{1}{2}}{1,1}{\frac{4}{k^2}},\]
and for $0<k\leq 2$,
\[\mathcal{J}(k) := \lim_{\varepsilon \to 0^{+}} \mathcal{J}(k + i \varepsilon),\] where we use the analytic continuation of $\mathcal{J}(z)$ to the complex upper half-plane. 
\end{prop}
\begin{proof} From \eqref{eq:Zk} and \eqref{eq:intS2S1}, we have, for all $k > 0,$ 
\[Z(s,x+y+k)=  k^s \int_0^2 
   \pFq{2}{1}{-\frac{s}{2},-\frac{s}{2}}{1}{\frac{v^2}{k^2}}
  \PP(S_1=v) \dd v =  k^s \cdot
  \pFq{3}{2}{-\frac{s}{2},-\frac{s}{2}, \frac{1}{2}}{1,1}{\frac{4}{k^2}}.\]
\end{proof}
For $z=\frac{k^2}{4}$, define
\[\tilde{\mathcal{J}}(z) :=(4z)^\frac{s}{2}\cdot \pFq{3}{2}{-\frac{s}{2},-\frac{s}{2}, \frac{1}{2}}{1,1}{\frac{1}{z}}.\]
Then $\tilde{\mathcal{J}}(z)$ satisfies a differential equation  
\[
s^3 \mathcal{U}(z) -2 (3 s^2 z+s (2-6 z)+4 z-2)\mathcal{U}'(z) - 4 z (-3 s z+s+6 z-5) \mathcal{U}''(z) + 8z^2(1-z)\mathcal{U}'''(z) = 0.
\]
The above differential equation has a regular singularity at $z = 0$ with local exponents $0, 0$, and $\frac{1+s}{2}$. Then, using Frobenius’ method to
solve the differential equation, we find the following basis of solutions around $z=0$: 
\[ z^\frac{1+s}{2}G_0(z;s),\qquad G_1(z;s), \qquad \text{and}\qquad G_2(z;s)+\log (z) G_1(z;s)\]
where $G_0, G_1$, and $G_2$ are holomorphic and non-zero at $z=0$, and $G_0$ and $G_1$ are given by \eqref{eq:Gs}
and $G_2(z;s)$ is the  Meijer $G$-function $G_{4,3}^{\,2,3}\!\left(\left.{\begin{matrix}
 1+\frac{s}{2}, 1+\frac{s}{2}, 1+\frac{s}{2}\\[3pt] 0, 0, \frac{1+s}{2}  \end{matrix}}\;\right|\,z\right)$.


We now have all the elements to prove Theorem \ref{thm:ZMMk+x+y}.

\begin{proof}[Proof of Theorem \ref{thm:ZMMk+x+y}]Since for almost all $s$, we have that $Z(s,x+y+k)$ converges as $k \rightarrow 0$, this eliminates the possibility of having 
$G_2(z;s)+\log (z) G_1(z;s)$ as a contributor to the expression of $Z(s,x+y+k)$, since it does not converge as $k\rightarrow 0$ and $z=\frac{k^2}{4}$. Therefore, we expect $Z(s,x+y+k)$ to be a linear combination of $z^\frac{1+s}{2}G_0(z;s)$ and $G_1(z;s)$. We consider the initial conditions $Z(s,x+y)=\frac{\Gamma(1+s)}{\Gamma\left(1+\frac{s}{2}\right)^2}$ and 
\[Z(s,x+y+1)=\frac{1}{2^{2s+1}}\tan \left(\frac{\pi s}{2}\right)\binom{s}{\frac{s-1}{2}}^2 \pFq{3}{2}{\frac{1}{2},\frac{1}{2}, \frac{1}{2}}{\frac{3+s}{2},\frac{3+s}{2}}{\frac{1}{4}}+\binom{s}{\frac{s}{2}}\pFq{3}{2}{-\frac{s}{2},-\frac{s}{2}, -\frac{s}{2}}{1,\frac{1-s}{2}}{\frac{1}{4}}\]
(see \cite[Corollary 1]{Borwein-Straub-Wan}). Thus, we conclude that 
\begin{equation}\label{ZMFchapt} Z(s,x+y+k)=\beta_0(s)\left(\frac{k}{2}\right)^{1+s} G_0\left(\frac{k^2}{4};s\right)+\beta_1(s) G_1\left(\frac{k^2}{4};s\right),\end{equation}
where 
\[\beta_0(s)=\frac{1}{2^{s}} \tan\left(\frac{\pi s}{2}\right) 
\left(\frac{\Gamma(1+s)}{\Gamma\left(\frac{1+s}{2}\right)\Gamma\left(\frac{3+s}{2}\right)}\right)^2
\qquad
\text{and}\qquad \beta_1(s)=\frac{\Gamma(1+s)}{\Gamma\left(1+\frac{s}{2}\right)^2}.\]
This concludes our argument. 
\end{proof}

\begin{prop}
   For $0\leq k \leq 2$ we have 
 \[      \m(x +y+k) =  \frac{k}{\pi}\cdot \pFq{3}{2}{\frac{1}{2},\frac{1}{2}, \frac{1}{2}}{\frac{3}{2},\frac{3}{2}}{\frac{k^2}{4}}.\]
\end{prop}
\begin{proof}
The Mahler measure of $x+y+k$ can be obtained by computing $\frac{\dd}{\dd s} Z(s, x+y+k) |_{s=0}$.
For $k \geq 2$, we find that $\m(x+y+k) = \log (k)$.
For $k < 2$, we differentiate the expression \eqref{ZMFchapt} with respect to $s$. We have that 
\begin{align*}
\beta_0(s)=&\frac{2s}{\pi}+\mathcal{O}(s^2),\\
\beta_1(s) =& 1+\mathcal{O}(s^2),\\
G_0\left(\frac{k^2}{4};s\right)=&\pFq{3}{2}{\frac{1}{2},\frac{1}{2}, \frac{1}{2}}{\frac{3}{2},\frac{3}{2}}{\frac{k^2}{4}}+\mathcal{O}(s),\\
G_1\left(\frac{k^2}{4};s\right)=&1+\mathcal{O}(s^3).
\end{align*}
Therefore, computing $\frac{\dd}{\dd s} Z(s, x+y+k) |_{s=0}$, we conclude that 
\[\m(x+y+k)=\frac{k}{\pi}\cdot \pFq{3}{2}{\frac{1}{2},\frac{1}{2}, \frac{1}{2}}{\frac{3}{2},\frac{3}{2}}{\frac{k^2}{4}}.\]
\end{proof}

\begin{proof}[Proof of Theorem \ref{thm:x+y+k}]
By \cite{wolfram-52} and \cite{wolfram-32}, we have for $|z|<1$, 
\begin{align*} 
\pFq{3}{2}{\frac{1}{2},\frac{1}{2}, \frac{1}{2}}{\frac{3}{2},\frac{3}{2}}{z}+\frac{16z}{9}\cdot \pFq{3}{2}{-\frac{1}{2},\frac{1}{2},\frac{3}{2}}{\frac{5}{2},\frac{5}{2}}{z}=&\frac{(2z+5)\sqrt{1-z}}{4}-\frac{(1-8z)\mathrm{Log}\left(\sqrt{1-z}+\sqrt{-z}\right)}{4\sqrt{-z}},
\end{align*}
where we are taking the principal branch of the logarithm. 
Notice that the above identity holds for $\{|z| < 1\} \setminus (0, \infty).$ However, it can  also be extended to $(0,1)$ since the function $\frac{\mathrm{Log}\left(\sqrt{1-z}+\sqrt{-z}\right)}{\sqrt{-z}}$ is continuous in this set. 

Comparing \eqref{mahlerarealkxy} and \eqref{eq:hypgeok+x+y}, we conclude that 
\begin{align*}
\m(x+y+k)-\m_{\D}(x+y+k)
=& \frac{k}{\pi}\cdot \pFq{3}{2}{\frac{1}{2},\frac{1}{2}, \frac{1}{2}}{\frac{3}{2},\frac{3}{2}}{\frac{k^2}{4}}+\frac{4 k^3}{9 \pi}\cdot \pFq{3}{2}{-\frac{1}{2},\frac{1}{2},\frac{3}{2}}{\frac{5}{2},\frac{5}{2}}{\frac{k^2}{4}}- \frac{k^2}{2} + \frac{1}{4}\\
=&\frac{k\sqrt{4-k^2}(10+k^2) +(8-16k^2)\arccos\left(\frac{k}{2}\right)}{16\pi },
\end{align*}
 which proves Theorem \ref{thm:x+y+k}. 
 \end{proof}

 Evaluating the above equality at $k = \sqrt{2},$ we have \[\m(x+y+\sqrt{2}) - \m_{\D}(x+y+\sqrt{2}) = \frac{3}{2\pi} - \frac{3}{8}.\] Now comparing formulas \eqref{eq:k=sqrt2} and \eqref{eq:k=sqrt2classical}, and combining with \cite[Theorem 6]{Lalin-Roy}, we conclude that  \[ \mathcal{C}_{\sqrt{2}} = \frac{\Gamma\left(\frac{3}{4}\right)^2}{\sqrt{2\pi^3}} \cdot 
 \pFq{4}{3}{\frac{1}{4},\frac{1}{4},\frac{3}{4},\frac{3}{4}}{\frac{1}{2},\frac{5}{4},\frac{5}{4}}{1} - \frac{\Gamma\left(\frac{1}{4}\right)^2}{72\sqrt{2\pi^3}} \cdot 
 \pFq{4}{3}{\frac{3}{4},\frac{3}{4},\frac{5}{4},\frac{5}{4}}{\frac{3}{2},\frac{7}{4},\frac{7}{4}}{1}
  = \frac{\log 2}{4},\] 
as remarked in the introduction.

 Further combining the identities \eqref{eq:k<2} and \eqref{CMZAMM} together yields a much simpler expression for $\m_{\D}(x+y+k)$ in terms of special values of Bloch--Wigner dilogarithm  for all $k \in \C,$ namely, \begin{align*}
     \m_{\D}(x+y+k) =& \frac{1}{\pi} \left( 2 \log (k) \arcsin\left(\frac{k}{2}\right) + D \left( \e^{2 i \arcsin\left(\frac{k}{2}\right)}\right)\right) \\ &-\frac{k\sqrt{4-k^2}(10+k^2) +(8-16k^2)\arccos\left(\frac{k}{2}\right)}{16\pi }.
 \end{align*}

\subsection{A direct computation for the areal Mahler measure of $x+y+k$}\label{sec:direct}

It is possible to prove Theorem \ref{thm:x+y+k} directly, without computing the areal Zeta Mahler function of $x+y+k$. More precisely, we get the following result.
\begin{thm}\label{thm:direct}
    For $k \geq 2,$ \[\m_{\D}(x+y+k) = \log (k),\] and, for $0 \leq k < 2,$
    \begin{align*}
    \m_{\D}(x+y+k) =& \frac{1}{\pi} \left( 2 \log(k) \arcsin\left(\frac{k}{2}\right)+ D \left( \e^{2 i \arcsin \left(\frac{k}{2}\right)}\right)\right)\\&  - \frac{k\sqrt{4-k^2}(10+k^2) +(8-16k^2)\arccos\left(\frac{k}{2}\right)}{16\pi }.
    \end{align*}
\end{thm}

To see this, we consider as before two random variables $X$ and $Y$ defined by walks of lengths $\rho_1$ and $\rho_2$ along the directions $\theta_1$ and $\theta_2,$ uniformly distributed on $[0, 1),$ respectively. As in the case of the areal Zeta Mahler function, we have that $X$ takes values $x = \rho_1 e^{2\pi i\theta_1}$ and $Y$ takes values $y=\rho_2 e^{2\pi i \theta_2}$. Consider the random variable $T_1 = |X+Y|$, and let $p_{T_1}(t)$ denote its probability density.  Then we have
\[
\int_{0}^2 p_{T_1}(t) t^s  \dd t= Z_{\D}(s,x+y).
\]
By Theorem \ref{thm:areal-zeta-x+y}, $Z_{\D} (s, x+y)=\frac{4}{s+4} \frac{\Gamma(2+s)}{\Gamma\left(\frac{s}{2}+2\right)^2}$. Applying the inverse Mellin transform to $Z_{\D}(s-1, x+y)$, we obtain, by Lemma \ref{densityt1},
\[p_{T_1}(t) = \PP(T_1 = t) = \frac{t}{\pi}\left(2\pi - t\sqrt{4-t^2} - 4\arcsin\left(\frac{t}{2}\right)\right).\]
\begin{proof}[Proof of Theorem \ref{thm:direct}]
Suppose that $k>0$. By Jensen's formula \eqref{eq:Jensen},  we
have the following identity \begin{align*}
    \m_{\D}(x+y+k) =& \frac{1}{2\pi^3i} \int_{\D^2 \times \TT} \log |x+y+kz| \dd A(x) \dd A(y) \frac{\dd z}{z}\nonumber \\ =& \log (k) + \frac{1}{\pi^2} \int_{\substack{\D^2 \\ |x+y| > k}} \log \left|\frac{x+y}{k}\right|\dd A(x) \dd A(y)\\
    =& \log (k) + \1_{\{0 <k < 2\}}\int_{k}^2 p_{T_1}(t)\log \left(\frac{t}{k}\right) dt,
\end{align*}
where $\1_{\{0< k < 2\}} = 1$ when $0 < k < 2,$ and equals $0$ otherwise. 

Therefore, we have, for $ k \geq 2,$ $\m_{\D}(x+y+k)= \log (k).$  For $0 < k < 2,$ we have \begin{align}
        \int_{k}^2 p_{T_1}(t) \log \left(\frac{t}{k}\right) \dd t =& \frac{1}{\pi}\int_{k}^2 t\log \left(\frac{t}{k}\right)\left(2\pi - t\sqrt{4-t^2} - 4\arcsin\left(\frac{t}{2}\right)\right) \dd t \nonumber \\ =& \frac{1}{\pi} \left[2\pi \int_{k}^2 t\log \left(\frac{t}{k}\right) \dd t - \int_{k}^2 t^2 \sqrt{4-t^2}\log \left(\frac{t}{k}\right) \dd t\right.\nonumber \\& \left.- 4\int_{k}^2 t\log \left(\frac{t}{k}\right) \arcsin\left(\frac{t}{2}\right) \dd t.  \right] \label{evaluate}
    \end{align}
We treat each integral separately. For example,
\begin{align*}
        2\int_{k}^2 \frac{\arcsin{(t/2)}}{t} \dd t =& \pi \log 2 - 2\arcsin\left(\frac{k}{2}\right)\log (k) - 2\int_{k}^2\frac{\log t}{\sqrt{4-t^2}} \dd t \nonumber \\ =&  \pi \log 2 - 2\arcsin\left(\frac{k}{2}\right)\log (k) - 2\int_{\arcsin\left(\frac{k}{2}\right)}^{\frac{\pi}{2}} \log (2\sin{\alpha}) \dd \alpha \nonumber \\ =& \pi \log 2 - 2\arcsin\left(\frac{k}{2}\right)\log (k) - D\left(e^{2i\arcsin\left(\frac{k}{2}\right)}\right), 
    \end{align*}
while the other integrals in \eqref{evaluate} can be evaluated by standard methods. Adding them all together, we reach the conclusion.
\end{proof}

\section{The areal Mahler measure of $(x+1)(y+1)+kz$}\label{sec:Qk}
In this section we prove Theorem \ref{thm:Qk}.
Recall that $Q_k = Q_k(x,y,z) =(x+1)(y+1)+kz$. Since $\m_{\D}(Q_k) = \m_{\D}(Q_{|k|})$, we will assume that $k$ is real and non-negative.

\subsection{The  Mahler measure of $(x+1)(y+1)+kz$} We start by computing  $\m(Q_k)$. We recall the following notation from Theorem \ref{thm:Qk}:\[F(t) = \frac{1}{2 \pi} \cdot \pFq{2}{1}{ \frac{1}{2}, \frac{1}{2}}{1}{1 - \frac{t^2}{16}}.\]

\begin{lem}
\label{Lem:ordinaryMM}
If $0< k <4$, then
\[
    \m((x+1)(y+1)+kz) = \log(k) + \int_k^4 \log\left(\frac{t}{k}\right) F(t)  \dd t
\]
and if $k\geq 4$,
$\m((x+1)(y+1)+kz) = \log(k)$. 
\end{lem}
\begin{proof}
Let $f(x,y) := (x + 1)(y + 1)$. By Jensen's formula \eqref{eq:Jensen}, we can write
\[
\m(Q_k) = \log(k) + \frac{1}{(2 \pi i)^2}\int_{\substack{\TT^2 \\ |f(x,y)| \geq k}} \log |f(x,y)/k| \, \frac{\dd x}{x} \frac{\dd y}{y},
\]
which returns $\log(k)$ if $k \geq 4$. Let $X$ and $Y$ be the random variables defined by walks of unit length along the directions $\theta_1$ and $\theta_2,$ uniformly distributed on $[0, 1),$ respectively. In particular, $X$ takes values $x = e^{2\pi i\theta_1}$ and $Y$ takes values $y=e^{2\pi i \theta_2}.$
Write $p(t)$ for the density function of $|f(X,Y)|$. For $a, b \in [0,4]$ we have that
\begin{equation*}
\frac{1}{(2 \pi i)^2} \int_{\substack{\TT^2 \\ a \leq |f(x,y)| \leq b}} |f(x,y)|^s \frac{\dd x}{x}\frac{\dd y}{y}= \int_{a}^{b} p(t) t^s  \dd t,
\end{equation*}
which is valid for $\re (s) > \sigma_1(f)$ for some negative $\sigma_1(f)$.
After differentiation with respect to $s$ on both sides
\begin{equation*}
\frac{1}{(2 \pi i)^2} \int_{\substack{\TT^2 \\ a \leq |f(x,y)| \leq b}} \log |f(x,y)| \cdot |f(x,y)|^s  \frac{\dd x}{x}\frac{\dd y}{y} = \int_{a}^{b} p(t) \log(t) t^s  \dd t.
\end{equation*}
For $0 < k < 4$ this implies that
\[
\m(Q_k) = \log(k) + \int_{k}^4 \log(t) p(t)  \dd t- \log(k) \int_{k}^4 p(t)  \dd t.
\]
It remains to apply the fact that $p(t) = F(t)$ by \cite[Proof of Thm 6.2]{BrunaultZudilin}. 
\end{proof}

\subsection{The  probability density function of the areal Mahler measure of $(x+1)(y+1)+kz$}
By Pritsker's Theorem \ref{thm:Pritsker}, for \emph{fixed} $x,y \in \D$,
\[
\frac{1}{\pi}\int_{\D} \log|Q_k(x,y,z)| \dd A(z) = \log(k) + \log^{+} |f(x,y)/k| + \frac{1}{2} (\min(|f(x,y)/k|^2,1) -1).
\]
This implies that
\begin{align}
\label{eq:mDQ}
\m_{\D}(Q_k) =& \frac{1}{\pi^2} \int_{\substack{\D^2 \\ |f(x,y)| \geq k}} \log|f(x,y)|  \dd A(x) \dd A(y) \nonumber \\&+ \frac{1}{2 \pi^2} \int_{\substack{\D^2 \\ |f(x,y)| \leq k}} \left( \frac{|f(x,y)|^2}{k^2} - 1 + 2 \log(k) \right)  \dd A(x) \dd A(y).    
\end{align}

As before, let $X$ and $Y$ be the random variables defined by walks of lengths $\rho_1$ and $\rho_2$ along the directions $\theta_1$ and $\theta_2,$ uniformly distributed on $[0, 1),$ respectively. In particular, $X$ takes values $x = \rho_1 e^{2\pi i\theta_1}$ and $Y$ takes values $y=\rho_2 e^{2\pi i \theta_2}.$
For the random walk $U = |f(X,Y)|$, we define $p_{U}(t)$  to be its probability density. Since $U$ takes values in $[0,4]$, the discussion of Remark \ref{probability} leads us  to
\[Z_{\D}(s,f(x,y))=
\int_{0}^4 p_{U}(t) t^s  \dd t.
\]
 $Z_{\D}(s,f(x,y))$ can be computed explicitly by \eqref{eq:LalRoy}, giving us 
 \begin{equation}
 \label{eq:zeta-f(x,y)}
      Z_{\D}(s,f(x,y)) = Z_{\D}(s,x+1)^2 = \frac{\Gamma(2+s)^2}{\Gamma\left(2+\frac{s}{2}\right)^4}.
 \end{equation}

Moreover, we have that  for $a, b \in [0,4]$,
\begin{equation}
\label{eq:moment}
\frac{1}{\pi^2} \int_{\substack{\D^2 \\ a \leq |f(x,y)| \leq b}} |f(x,y)|^s \dd A(x) \dd A(y) = \int_{a}^{b} p_{U}(t) t^s  \dd t,
\end{equation}
which is valid for $\re (s) > \sigma_1(f)$ for some negative $\sigma_1(f)$.
After differentiation with respect to $s$ on both sides
\begin{equation}
\label{eq:logmoment}
\frac{1}{\pi^2} \int_{\substack{\D^2 \\ a \leq |f(x,y)| \leq b}} \log |f(x,y)| \cdot |f(x,y)|^s  \dd A(x) \dd A(y) = \int_{a}^{b} p_{U}(t) \log(t) t^s  \dd t.
\end{equation}

The two equations \eqref{eq:moment} and \eqref{eq:logmoment} allow us to write the areal Mahler measure \eqref{eq:mDQ} in terms of $p_{U}$. We combine this in the following result.
\begin{lem}
\label{lem:mdform}
For $k \geq 4$
 \[
    \m_{\D}(Q_k) = \frac{9}{8k^2} - \frac{1}{2} + \log(k).
    \]
and for $0 < k < 4$
\[
\m_{\D}(Q_k) = \frac{9}{8k^2} - \frac{1}{2} + \log(k) + \int_{k}^4 p_{U}(t) \left( \log(t) - \frac{t^2}{2 k^2} + \frac{1}{2} - \log(k) \right) \dd t.
\]
\end{lem}
\begin{proof}
    For $0 \leq k < 4$, we apply \eqref{eq:mDQ}, \eqref{eq:moment}, and \eqref{eq:logmoment} to write
        \[
    \m_{\D}(Q_k) = \int_{k}^4 p_{U}(t) \log(t)  \dd t+ \frac{1}{2} \int_{0}^k \left(\frac{t^2}{k^2} - 1 + 2\log(k) \right) p_{U}(t)  \dd t.
    \]
We can put the two integrals together by observing that equation \eqref{eq:zeta-f(x,y)} yields
\[
\int_{0}^4 t^2 p_{U}(t)  \dd t= \frac{\Gamma(2+2)^2}{\Gamma\left(2+\frac{2}{2}\right)^4} = \frac{9}{4}
\]
and
\[
\int_{0}^4 p_{U}(t)  \dd t= \frac{\Gamma(2+0)^2}{\Gamma\left(2+0\right)^4} = 1,
\]
so that
\[
\m_{\D}(Q_k) = \frac{9}{8k^2} - \frac{1}{2} + \log(k) + \int_{k}^4 p_{U}(t) \left( \log(t) - \frac{t^2}{2 k^2} + \frac{1}{2} - \log(k) \right)  \dd t.
\]
We treat the range $k \geq 4$. Since $|f(x,y)| \leq 4$ for all $x, y \in \D$, by \eqref{eq:mDQ},
\begin{align*}
 \m_{\D}(Q_k) &=  \frac{1}{2 \pi^2} \int_{\D^2} \left( \frac{|f(x,y)|^2}{k^2} - 1 + 2 \log(k) \right)  \dd A(x) \dd A(y) \\
 &= \frac{1}{2 k^2} Z_{\D}(2,f(x,y)) - \frac{1}{2} + \log(k) = \frac{9}{8 k^2} - \frac{1}{2} + \log(k),
\end{align*}
where in the last identity, we have applied \eqref{eq:zeta-f(x,y)}. 
\end{proof}

We now consider the solutions $y(t)$ to the nonhomogeneous linear differential equation
\begin{equation}
\label{E1}
\theta^2 y(t) = t F(t).     
\end{equation}
Here $\theta$ is the differential operator $t \frac{\dd}{\dd t}$. 
A basis for the kernel of $\theta^2$ is given by $\{1, \log(t)\}$, so any two solutions of \eqref{E1} differ by a linear combination of $1$ and $\log(t)$. Since $F(t)$ is analytic in $t = 4$, we must have that $y(t)$ is analytic at this point. Therefore, we can find a unique solution $y_0(t)$ to \eqref{E1} with $y_{0}(4) = y_{0}'(4) = 0$. 
\begin{lem}
\label{Lem26} For $0\leq t \leq 4$, we have
    \[p_{U}(t) = t y_0(t).\]
\end{lem}
\begin{proof}
    It suffices to show that the Mellin transforms of  $p_{U}(t)$ and $t y_0(t)$ coincide.
    Since
    \[
    \int_0^\infty t^{s-1} p_{U}(t)  \dd t = Z_{\D}(s-1,x+1)^2 = \frac{\Gamma(1+s)^2}{\Gamma\left(\frac{3+s}{2}\right)^4},
    \]
    it suffices to show that
    \[
    \int_0^4 t^s y_0(t)  \dd t = \frac{\Gamma(1+s)^2}{\Gamma\left(\frac{3+s}{2}\right)^4}
    \]
    for $s$ with sufficiently big $\re(s)$.
    We will perform integration by parts twice and use the identity \cite[Proof of Thm 6.2]{BrunaultZudilin}:
    \begin{equation} \label{eq:bz}\int_0^4 t^s F(t) \dd t = \frac{\Gamma(1+s)^2}{\Gamma\left(1+\frac{s}{2}\right)^4} .\end{equation}
The first integration by parts gives
    \begin{align*}
        \int_0^4 t^s y_0(t)  \dd t &= \Big [ \frac{t^{s+1}}{s+1} y_0(t) \Big]_{0}^4 - \frac{1}{s+1}\int_0^4 t^s \theta y_0(t)  \dd t
        \\
        &= - \frac{1}{s+1}\int_0^4 t^s \theta y_0(t) \dd t.
    \end{align*}
Here the last line follows from $y_0(4) = 0$ and the fact that $s$ is chosen with $\re(s)$  big enough (so that $t^s y_0(t)$ is guaranteed to vanish at $t =0$).
Finally performing integration by parts a second time and using that  $y_0'(4) = 0$ gives
\begin{align*}
\int_0^4 t^s y_0(t)  \dd t &= \frac{1}{(s+1)^2} \int_0^4 t^s \theta^2 y_0(t)  \dd t \\ &= \frac{1}{(s+1)^2} \int_0^4 t^{s+1} F(t)  \dd t = \frac{1}{(s+1)^2} \frac{\Gamma(2+s)^2}{\Gamma\left(\frac{3+s}{2}\right)^4} =\frac{\Gamma(1+s)^2}{\Gamma\left(\frac{3+s}{2}\right)^4},
\end{align*}
by \eqref{eq:bz}.
This is what we needed to show. 
\end{proof}

\begin{lem}
\label{L10}
We have
    \[\theta y_0(t) =  -\int_t^4 F(u)  \dd u .\]
\end{lem}
\begin{proof}
Both the left and the right hand side are annihilated by the differential operator $\theta - t F$ and are analytic in a neighborhood of $t=4$, taking the value $0$ at this point. Since there is a unique function with these properties, the equality follows.
\end{proof}
\begin{lem}
\label{L11}
    We have
    \[y_0(t) = \int_t^4 \log(u) F(u)  \dd u  + \log(t) \theta y_0(t).\]
\end{lem}
\begin{proof}
If we apply the $\theta$ operator to the right-hand side we find 
\[
\theta \left( \int_t^4 \log(u) F(u) \dd u  + \log(t) \theta y_0(t)  \right) = - t \log(t) F(t) + \theta y_0(t) + \log(t) \theta^2 y_0(t) = \theta y_0(t).
\]
Also the value at $t = 4$ coincides on both sides, so the equality follows.
\end{proof}

\begin{coro}
\label{pdcor}
We have the following evaluations for $0 < k < 4$:
\begin{align*}
    p_{U}(k)/k &=  \int_{k}^{4} \log\left(\frac{u}{k}\right) F(u)  \dd u = \m(Q_k) - \log(k) \\
    \theta (p_{U}(t)/t)(k) &= -\int_k^4 F(u)  \dd u.
\end{align*}
\end{coro}
\begin{proof}
    For the first line we use Lemmas \ref{Lem:ordinaryMM}, \ref{Lem26}, \ref{L10}, and \ref{L11}. The second line is immediate from Lemma \ref{L10}. 
\end{proof}
\subsection{The Proof of Theorem \ref{thm:Qk}}
We recall the following notation from Theorem \ref{thm:Qk}
\begin{align*}
G(t) = \frac{1}{2 \pi }\pFq{2}{1}{\frac{1}{2}, -\frac{1}{2}}{1}{1-\frac{t^2}{16}},
\end{align*}
and for $0<k<4$ and $n$ a non-negative integer, we define $c_n(k)$ and $d_n(k)$ as
\begin{align*}
c_n(k) :=& \int_{k}^4 t^{2n} F(t) \dd t,\\
d_n(k) :=& \int_{k}^4 t^{2n} \log(t) F(t) \dd t. 
\end{align*}
\begin{lem}
\label{Lem:Lem34}
    Using the above notation, we have the following simplification for $\m_{\D}(Q_k)$
    in the range $0 < k < 4$: \[
\m_{\D}(Q_k) = \frac{9}{8k^2} + \log (k) - \frac{1}{2} +
\frac{k^2}{8} \m(Q_k) - \frac{k^2}{8} \log(k) - \frac{1}{32k^2}c_2(k) - \left(\frac{1}{8} +\frac{1}{4} \log(k) \right) c_1(k) + \frac{5k^2}{32} c_0(k) + \frac{1}{4} d_{1}(k).
\]
\end{lem}
\begin{proof}
Using Lemma \ref{lem:mdform}, we can write
\[
\m_{\D}(Q_k) = \frac{9}{8k^2} - \frac{1}{2} + \log(k) + \int_{k}^4 p_{U}(t) \left( \log(t) - \frac{t^2}{2 k^2} + \frac{1}{2} - \log(k) \right) \dd t.
\]
Our goal is to simplify the integral using the differential equation \eqref{E1} for $p_{U}(t)$.
We perform integration by parts 
\begin{align*}
   & \int_{k}^4 \left(t \log(t)- \frac{t^3}{2k^2} + \frac{t}{2} - t\log(k) \right) p_{U}(t)/t  \dd t\\
&= \left. \left(\frac{t}{2} \log (t) -\frac{t^3}{8 k^2} - \frac{t}{2} \log (k)  \right)p_{U}(t) \right|^4_k - \int_{k}^4 \left(\frac{t}{2} \log (t) -\frac{t^3}{8 k^2} - \frac{t}{2} \log (k)  \right) \theta (p_{U}(t)/t)  \dd t\\
&= \frac{k}{8} p_{U}(k) - \int_{k}^4 \left(\frac{1}{2} t \log (t) -\frac{t^3}{8 k^2} - \frac{1}{2} t \log (k)  \right) \theta (p_{U}(t)/t)  \dd t,
    \end{align*}
where we have used Lemma \ref{Lem26} which implies that $p_\D(4)=0$. 
We perform integration by parts another time
\begin{align*}
   & \int_{k}^4 \left(\frac{t}{2} \log (t) -\frac{t^3}{8 k^2} - \frac{t}{2} \log (k)  \right) \theta (p_{U}(t)/t) \dd t\\ &= \left. \left( \frac{t^2}{4} \log(t) - \frac{t^2}{8} - \frac{t^4}{32 k^2}  - \frac{t^2}{4} \log(k) \right) \theta(p_{U}(t)/t) \right|^4_k -\int_{k}^4 \left( \frac{t}{4} \log(t) - \frac{t}{8} - \frac{t^3}{32 k^2}  - \frac{t}{4} \log(k) \right) \theta^2(p_{U}(t)/t)  \dd t \\
   &= \frac{5k^2}{32} \theta (p_{U}(t)/t)(k) - \int_{k}^4 \left( \frac{t}{4}\log(t) - \frac{t}{8} - \frac{t^3}{32 k^2}  - \frac{t}{4} \log(k) \right) t F(t)  \dd t\\
   &= \frac{5k^2}{32}  \theta (p_{U}(t)/t)(k) - \frac{1}{4}d_1(k) + \frac{1}{32k^2}c_2(k) + \left(\frac{1}{8} +\frac{1}{4} \log(k) \right) c_1(k),   
\end{align*}
where we have used Lemma \ref{Lem26} as well as Lemma \ref{L10} and the fact that $y_0'(4)=0$.
Combining all this and applying Corollary \ref{pdcor}, we find
\begin{align*}
\m_{\D}(Q_k) &= \frac{9}{8k^2} - \frac{1}{2} + \log(k) + \frac{k}{8}p_{U}(k) - \frac{1}{32k^2} c_2(k) -\left(\frac{1}{8} +\frac{1}{4} \log(k) \right) c_1(k) - \frac{5k^2}{32}\theta (p_{U}(t)/t)(k) + \frac{1}{4}d_{1}(k)   \\
 &= \frac{9}{8k^2} - \frac{1}{2} + \log(k) + \frac{k^2}{8}\m(Q_k) - \frac{k^2}{8}\log(k) - \frac{1}{32k^2} c_2(k) -\left(\frac{1}{8} +\frac{1}{4} \log(k) \right) c_1(k) + \frac{5k^2}{32}c_0(k)\\& \qquad + \frac{1}{4}d_1(k). 
\end{align*}
This is what we needed to show.
\end{proof}
    
Next, we wish to simplify the expression in Lemma \ref{Lem:Lem34}. We do this using the following lemma.
\begin{lem} \label{lem:crazymatrix}
    The following identities hold:
    \begin{align*}
0=&\frac{20}{k}c_0(k)-\frac{3(k^2+12)}{k^3}c_1(k)+\frac{4}{k^3}c_2(k)+\frac{8}{k}d_0(k)-\frac{2}{k}d_1(k)+(8-k^2) \log(k) F(k)+8\log(k)G(k),\\
0=&\frac{120}{k^3}c_0(k)-\frac{30(k^2+12)}{k^5}c_1(k)+\frac{40}{k^5}c_2(k)-5F(k)+ \frac{80}{k^2}G(k),\\
0=&\frac{8}{k}c_0(k)-\frac{2}{k}c_1(k)+(8-k^2)F(k)+8G(k).
\end{align*}  
\end{lem}
\begin{proof}
We first introduce the following notation
\begin{align*}
c'_n(k):=& \int_{k}^4 t^{2n} G(t) \dd t, \\
d'_n(k) :=& \int_{k}^4 t^{2n} \log(t) G(t) \dd t.
\end{align*}

    We will make use of the following identities
    \[
    \frac{\dd}{\dd t} \left( 2 \cdot \pFq{2}{1}{\frac{1}{2}, -\frac{1}{2}}{1}{t}
    + 2(t - 1)  \cdot 
    \pFq{2}{1}{\frac{1}{2}, \frac{1}{2}}{1}{t}
   \right) = \pFq{2}{1}{\frac{1}{2}, \frac{1}{2}}{1}{t}
    \]
    and
\[
    \frac{\dd}{\dd t} \left( \frac{2}{3}(t + 1) \cdot 
    \pFq{2}{1}{\frac{1}{2}, -\frac{1}{2}}{1}{t}
     + \frac{2}{3}(t - 1)  \cdot 
     \pFq{2}{1}{\frac{1}{2}, \frac{1}{2}}{1}{t}
    \right) =\pFq{2}{1}{\frac{1}{2}, -\frac{1}{2}}{1}{t}.
    \]

Applying integration by parts and employing the above identities, we can directly 
derive the following four relations:
\begin{align*}
\frac{c_n(k)}{k^{2n+1}} &= \frac{16}{k^2} G(k) - F(k) - \frac{2n - 1}{k^{2n+1}}(c_n(k) - 16c'_{n-1}(k)), \\
\frac{d_n(k)}{k^{2n+1}} &= \frac{16c'_{n-1}(k) - c_n(k)}{k^{2n+1}} - \frac{2n-1}{k^{2n+1}}(d_n(k) - 16 d'_{n-1}(k))+\log(k)\left(\frac{16}{k^2}G(k)-F(k)\right),\\
\frac{c_n'(k)}{k^{2n+1}} &= \frac{1}{3}\left(\frac{32}{k^2}-1\right) G(k) -\frac{1}{3} F(k) - \frac{2n - 1}{3k^{2n+1}}(c_n(k) +c'_n(k) -32c'_{n-1}(k)),\\
\frac{d_n'(k)}{k^{2n+1}} &= 
-\frac{1}{3k^{2n+1}}(c_n(k) +c'_n(k) -32c'_{n-1}(k))-\frac{2n-1}{3k^{2n+1}}(d_n(k) +d'_n(k) -32d'_{n-1}(k))\\
&+\frac{\log(k)}{3}\left[\left(\frac{32}{k^2}-1\right) G(k) - F(k)\right].\\
\end{align*}
From these, we collect all linear relations containing only elements from
\[
\{c_0(k), c_1(k), c_2(k), d_0(k), d_1(k), c'_{-1}, c'_0(k), c'_1(k), c'_2(k),d'_{-1}(k), d'_0(k), d'_1(k), F(k), G(k)\}.
\]
This gives a total of ten linear relations, which we can summarize in a matrix (keeping the logical order):
\[
B = \left(\begin{array}{rrrrrrrrrrrrrr}
0 & 0 & 0 & 0 & 0 & -\frac{16}{k} & 0 & 0 & 0 & 0 & 0 & 0 & -1 & \frac{16}{k^2}\\[6pt]
0 & -\frac{2}{k^3} & 0 & 0 & 0 & 0 & \frac{16}{k^3} & 0 & 0 & 0 & 0 & 0 & -1 & \frac{16}{k^2}\\[6pt]
0 & 0 & -\frac{4}{k^5} & 0 & 0 & 0 & 0 & \frac{48}{k^5} & 0 & 0 & 0 & 0 & -1 & \frac{16}{k^2}\\[6pt]
-\frac{1}{k} & 0 & 0 & 0 & 0 & \frac{16}{k} & 0 & 0 & 0 & -\frac{16}{k} & 0 & 0 & -\log(k) & \log(k) \frac{16}{k^2}\\[6pt]
0 & -\frac{1}{k^3} & 0 & 0 & -\frac{2}{k^3} & 0 & \frac{16}{k^3} & 0 & 0 & 0 & \frac{16}{k^3} & 0 & -\log(k) & \log(k) \frac{16}{k^2}\\[6pt]
\frac{1}{3 k} & 0 & 0 & 0 & 0 & -\frac{32}{3 k} & -\frac{2}{3k} & 0 & 0 & 0 & 0 & 0 & -\frac{1}{3} & \frac{1}{3} \left(\frac{32}{k^2} - 1\right)\\[6pt]
0 & -\frac{1}{3k^3} & 0 & 0 & 0 & 0 & \frac{32}{3k^3} & -\frac{4}{3k^3} & 0 & 0 & 0 & 0 & -\frac{1}{3} & \frac{1}{3} \left(\frac{32}{k^2} - 1\right)\\[6pt]
0 & 0 & -\frac{1}{k^5} & 0 & 0 & 0 & 0 & \frac{32}{k^5} & \frac{-2}{k^5} & 0 & 0 & 0 & -\frac{1}{3} & \frac{1}{3} \left(\frac{32}{k^2} - 1\right)\\[6pt]
-\frac{1}{3k} & 0 & 0 & \frac{1}{3k} & 0 & \frac{32}{3k} & -\frac{1}{3k} & 0 & 0 & -\frac{32}{3k} & -\frac{2}{3k} & 0 & -\frac{\log(k)}{3} & \frac{\log(k)}{3} \left(\frac{32}{k^2} - 1 \right)\\[6pt]
0 & -\frac{1}{3 k^3} & 0 & 0 & -\frac{1}{3k^3} & 0 & \frac{32}{3 k^3} & -\frac{1}{3k^3} & 0 & 0 & \frac{32}{3k^3} & -\frac{4}{3k^3} & -\frac{\log(k)}{3} & \frac{\log(k)}{3} \left(\frac{32}{k^2} - 1 \right)
\end{array}\right).
\]
Now the identities can be proved using linear algebra. For instance if \[v_1 = (-24, 24 + k^2, -k^2, -16, k^2, 36, -36, 0, 24, 0),\] an easy computation shows 
\[
v_1 B = \left(\frac{20}{k},-3\frac{12+k^2}{k^3},\frac{4}{k^3},\frac{8}{k},-\frac{2}{k},0,0,0,0,0,0,0,(8 - k^2) \log(k),8 \log(k) \right).\]
Which proves the first identity. The other two identities can be proved using the vectors
\[
v_2 = \left ( -\frac{240}{k^2},\frac{15 \left(k^2+16\right)}{k^2},-10,0,0,\frac{360}{k^2},-\frac{360}{k^2},0,0,0\right )
\]
and
\[
v_3 = \left (-16,k^2,0,0,0,24,0,0,0,0\right)
\]
respectively. 

\end{proof}

\begin{proof}[Proof of Theorem \ref{thm:Qk}]
Using the identities from Lemma \ref{lem:crazymatrix} as well as the fact that $\m(Q_k)=(1-c_0(k))\log(k)+d_0(k)$ for $0<k<4$, which follows from Lemma \ref{Lem:ordinaryMM}, the result of Lemma \ref{Lem:Lem34} simplifies to 
\begin{align*}
    \m_{\D}(Q_k) = &\frac{9}{8k^2} - \frac{1}{2} +
\frac{k^2 + 8}{8} \m(Q_k) - \frac{k^2}{8}\log(k) + c_0(k)\left(\frac{1}{2} - \frac{9}{8 k^2} + \frac{5 k^2}{32} \right)\\
    &+ F(k) \left(-\frac{9}{8 k} - \frac{29 k}{64} + \frac{17 k^3}{128} \right) + G(k) \left(-\frac{9}{8 k} - \frac{49 k}{32}\right),
\end{align*}
which completes the proof of Theorem \ref{thm:Qk}.
\end{proof}

\kommentar{\begin{thm} 
If $k\geq 4$, we have 
    \[
   \m_{\D}(Q_k) - \frac{k^2 + 8}{8} \m (Q_k)= \frac{9}{8k^2} - \frac{1}{2} - \frac{k^2}{8} \log(k).
    \]
If $0<k<4$, we have 
    \begin{align}
    \m_{\D}(Q_k) - \frac{k^2 + 8}{8} \m (Q_k)=&\frac{9}{8k^2} - \frac{1}{2} - \frac{k^2}{8}\log(k) + c_0(k)\left(\frac{1}{2} - \frac{9}{8 k^2} + \frac{5 k^2}{32} \right) \label{eq:34}\\
    &+ F(k) \left(-\frac{9}{8 k} - \frac{29 k}{64} + \frac{17 k^3}{128} \right) + G(k) \left(-\frac{9}{8 k} - \frac{49 k}{32}\right). \nonumber
    \end{align}
\end{thm}}
\section{Further understanding of the main results}\label{sec:vol}
In this section we provide some further discussions on the formulas that were obtained in Theorems \ref{thm:x+y+k} and \ref{thm:Qk}. 

\subsection{Volumes of Deninger cycles}
An important observation about Theorem \ref{thm:Qk} is that formula \eqref{eq:34} for the case $0<k<4$ contains the term $c_0(k)$,  that is not likely to simplify to an elementary or hypergeometric function. However $c_0(k)$ has a geometric interpretation as a volume associated to the Deninger cycle (see page 102 in \cite{BrunaultZudilin}). 
\begin{prop}
\label{prop:Deninger}
    Let $0 < k < 4$ and let $D_{Q_k, z} = \{ |x| = |y| = 1, |z| \geq 1  \, :\,  Q_k(x,y,z) = 0 \}$ be the Deninger cycle of $Q_k(x,y,z)$ with respect to $z$. Then $c_0(k) = \textnormal{vol}(D_{Q_k, z})$, that is, the volume (area) of $D_{Q_k, z}$.
\end{prop}
\begin{proof}
    Using \cite[Proof of Thm 6.2]{BrunaultZudilin}, it follows that $F(t)$ is the probability of the random walk $(x+1)(y+1)$, where $X$ and $Y$ are independent random variables uniformly distributed on the complex unit circle. 
    We can rewrite $D_{Q_k, z} = \{|x| = |y| = 1 \, : \, k \leq |(x+1)(y+1)| \}$. It follows that
    \[
    \textnormal{vol}(D_{Q_k, z}) = \frac{1}{(2 \pi i)^2}\int_{\substack{\TT^2 \\ k \leq |(x+1)(y+1)|}} \frac{\dd x}{x} \frac{\dd y}{y} = \int_{k}^4 F(t) \, \dd t = c_0(k).
    \]
This completes the proof.
\end{proof}

If we compute the volumes in Proposition \ref{prop:Deninger} of the Deninger cycles $D_{Q_k,x}$ and $D_{Q_k,y}$ with respect to the other variables we get a similar result.

\begin{prop}
    Using the notation from Proposition \ref{prop:Deninger}. We have 
    \[
    \textnormal{vol}(D_{Q_k, x}) = \textnormal{vol}(D_{Q_k, y}) = 1 - \frac{1}{2} c_0(k).
    \]
\end{prop}
\begin{proof}
    By the symmetry of $Q_k$, it suffices to compute $\textnormal{vol}(D_{Q_k, x})$. We have
    \[D_{Q_k,x} = \{ |y| = 1, \, |z| = 1 \,:\, \left| 1 + \frac{k z}{1+y}\right| \geq 1\} .\]
Note that $\left| 1 + \frac{k z}{1+y}\right| \geq 1$ is equivalent to $\left| 1 + y + k z\right| \geq \left| 1 + y \right|$, which is equivalent to $z + \frac{1}{z} + \frac{y}{z} + \frac{z}{y} \geq - k$, since $y$ and $z$ lie on the unit circle.
 After a change of variables $v = \frac{z}{y}$, we can write 
\begin{align*}
    \textnormal{vol}(D_{Q_k,x}) = \frac{1}{(2 \pi i)^2} \int_{\substack{\TT^2 \\-k \leq z + \frac{1}{z} + v + \frac{1}{v}}} \frac{\dd z}{z} \frac{\dd v}{v} &= \frac{1}{(2 \pi i)^2}\left(\int_{\substack{\TT^2 \\ -k \leq z + \frac{1}{z} + v + \frac{1}{v} \leq k}} + \int_{\substack{\TT^2 \\ k \leq z + \frac{1}{z} + v + \frac{1}{v}}} \right) \frac{\dd z}{z} \frac{\dd v}{v}\\
    &= \frac{1}{(2 \pi i)^2}\left(\int_{\substack{\TT^2 \\ \left| z + \frac{1}{z} + v + \frac{1}{v} \right| \leq k}} + \frac{1}{2} \int_{\substack{\TT^2 \\ k \leq \left| z + \frac{1}{z} + v + \frac{1}{v} \right|}} \right)  \frac{\dd z}{z} \frac{\dd v}{v}\\
    &= 1 - \frac{1}{2} \cdot \frac{1}{(2 \pi i)^2}\int_{\substack{\TT^2 \\ k \leq \left| z + \frac{1}{z} + v + \frac{1}{v} \right|}} \frac{\dd z}{z} \frac{\dd v}{v}.
\end{align*}
The second line follows from the anti-symmetry of  $z + \frac{1}{z} + v + \frac{1}{v}$:  $k \leq z + \frac{1}{z} + v + \frac{1}{v}$ if and only if  $(-z) + \frac{1}{(-z)} + (-v) + \frac{1}{(-v)} \leq -k$.
Next, it remains to notice that
\[\frac{1}{(2 \pi i)^2} \int_{\substack{\TT^2 \\ k \leq \left| z + \frac{1}{z} + v + \frac{1}{v} \right|}} \frac{\dd z}{z} \frac{\dd v}{v} = \frac{1}{(2 \pi i)^2} \int_{\substack{\TT^2 \\ k \leq \left|(1+u)(1+w) \right|}} \frac{\dd u}{u} \frac{\dd w}{w} = \textnormal{vol}(D_{Q_k,z}) = c_0(k),\]
again with a change of variables $u = zv, w = \frac{z}{v}$, and for the last equality we used Proposition \ref{prop:Deninger}.
We conclude with \[\textnormal{vol}(D_{Q_k, x})  = 1 - \frac{1}{2} c_0(k). \]
\end{proof}


Similarly, the term $\frac{\arccos(\frac{k}{2})}{\pi}$ in Theorem \ref{thm:x+y+k} can be interpreted as a volume.
\begin{prop}
For $0 < k \leq 2$ and let $D_{x+y+k, y} = \{ |x|=1, |y| \geq 1  \, : \, x+y+k= 0 \}$ be the Deninger cycle of $x+y+k$ with respect to $y$, and define $D_{x+y+k, x}$ similarly with respect to $x$. Then 
    \[ \textnormal{vol}(D_{x+y+k,x}) =  \textnormal{vol}(D_{x+y+k,y}) = 1 - \frac{\arccos(\frac{k}{2})}{\pi}. \]
\end{prop}
\begin{proof}
    By symmetry it suffices to compute $\textnormal{vol}(D_{x+y+k,x})$. We have that
    \[
    D_{x+y+k,x} = \{|y| = 1 \,: \, |y+k | \geq 1\}.
    \]
Note that  $|y + k| \geq 1$ is equivalent to $y + \frac{1}{y} \geq -k$. Upon writing $y = e^{2 \pi i t}$, it follows that
\begin{align*}
\textnormal{vol}(D_{x+y+k,x}) = \frac{1}{2 \pi i }\int_{\substack{\TT \\ |k + y| \geq 1}} \frac{\dd y}{y} &= \int_{\substack{0 \leq t \leq 1 \\ \cos(2 \pi t) \geq -\frac{k}{2}}} \dd t \\
&=\frac{\arccos(-\frac{k}{2})}{\pi} = 1 - \frac{\arccos(\frac{k}{2})}{\pi}.
\end{align*}
\end{proof}

\subsection{Modular interpretation of $c_0(k)$}

The aim of this subsection is to gain further understanding of the quantity $c_0(k)$ using modular forms.  
Let $q = e^{2 \pi i \tau}$.  Then the Dedekind eta function is defined by
\[
\eta(\tau) = e^{\pi i \tau/12} \prod_{n = 1}^\infty (1 - q^n).
\]
We consider \[x(\tau) = 
16 \left( \frac{\eta(2 \tau) \eta(8 \tau)^2}{\eta(4 \tau)^3} \right)^4 = 16q - 64q^3 + 224q^5 - 640q^7 + 1616q^9 + \cdots.
\]
This is a modular function for the group $\Gamma_1(8)$ and it maps the imaginary interval $(+ i \infty, 0)$ bijectively to the real interval $(0,4)$, see \cite[p. 437]{StraubZudilin}. Our result is the following.
\begin{thm}
\label{thm:modular}
Let $0 < k < 4$ and suppose $\tau_k = i t_k  \in i \R_{>0}$ is such that $x(-1/(4\tau_k)) = k$. Then
\[
c_0(k) = \frac{8 i}{\pi^3} \sum_{\substack{m, n \in \Z \\ m, n \textnormal{ odd}}}\frac{e^{-2 \pi i m/4}}{(m\tau_k + n)^2 m}.
\]
\end{thm}
\begin{ex}
We apply Theorem \ref{thm:modular} to $\tau_k = i$. 
Using the special values of the Dedekind eta functions
\cite[p. 326]{Andrews} 
\[
\eta(i/2) = \frac{\Gamma(1/4)}{2^{7/8} \pi^{3/4}}; \qquad \eta(i) = \frac{\Gamma(1/4)}{2 \pi^{3/4}};\qquad \eta(2i)  = \frac{\Gamma(1/4)}{2^{11/8} \pi^{3/4}},
\]
we find that $x(i/4) = \sqrt{8}$. 
Theorem \ref{thm:modular} implies that
\[
c_0(\sqrt{8}) =  \frac{8 i}{\pi^3} \sum_{\substack{m, n \in \Z \\ m, n \textnormal{ odd}}}\frac{e^{-2 \pi i m/4}}{(m i + n)^2 m}.
\]
However, we are not aware of further simplifications in this case.
\end{ex}
\begin{rem} 
A priori, we need to take into consideration the order of summation in Theorem \ref{thm:modular}. However, the next lemma shows that the sum is independent of the order. 
\end{rem}

\begin{lem}
\label{Lem:absconv}
For each $\tau$ in the complex upper half-plane, the sum 
    \[\sum_{\substack{m, n \in \Z \\ m, n \textnormal{ odd}}}\frac{e^{-2 \pi i m/4}}{(m\tau+n)^2 m}\]
    converges absolutely.
\end{lem}

\begin{proof}
Using \cite[Lemma 4.2]{Conrad}, there exists a $\delta = \delta_{\tau} \in (0,1)$ such that
\[
|m \tau + n| \geq \delta |m i + n|.
\]
We find that
\[
\sum_{\substack{m, n \in \Z \\ m, n \textnormal{ odd}}}\left| \frac{e^{-2 \pi i m/4}}{(m\tau+n)^2m} \right| \leq \frac{1}{\delta^2}\sum_{\substack{m,n \in \Z \\ m \neq 0}} \frac{1}{|m i + n|^2 |m| } = \frac{1}{\delta^2}\sum_{\substack{m,n \in \Z \\ m \neq 0}} \frac{1}{(m^2 + n^2) |m|}.
\]
It remains to show that
\[
\sum_{\substack{m,n \in \Z \\ m \neq 0}} \frac{1}{(m^2 + n^2) |m|} = 2 \zeta(3) + 4 \sum_{\substack{m,n \geq 1}} \frac{1}{(m^2+n^2)m}
\]
converges.  Notice that all the terms in the above sum are positive, and therefore, by the Fubini--Tonelli Theorem, the sum can be performed in either order and it suffices to prove that the final result is finite. By the Weierstrass product of the hyperbolic sine function, we have  \[\frac{\sinh{\pi z}}{\pi z} = \prod_{n \geq 1} \left(1+\frac{z^2}{n^2}\right).\] Taking the logarithm and differentiating both sides, we get \[\pi \coth{\pi z} - \frac{1}{z} = \sum_{n \geq 1} \frac{2z}{z^2 + n^2}.\]  Replacing $z = m,$ the above equality yields that \[\pi\frac{ \coth{\pi m}}{m} - \frac{1}{m^2} = \sum_{n \geq 1} \frac{2}{m^2 + n^2}.\]  Therefore, 
    \[\sum_{m, n \geq 1} \frac{1}{(m^2 + n^2)m} = - \sum_{m \geq 1} \frac{1}{2m^3} + \frac{\pi}{2}\sum_{m \geq 1} \frac{\coth{ \pi m}}{m^2} = -\frac{\zeta(3)}{2} + \frac{\pi}{2} \sum_{m \geq 1} \left(\frac{e^{2\pi m} +1}{e^{2\pi m} - 1}\right)\frac{1}{m^2},\] and since $\lim_{m \rightarrow \infty}\frac{e^{2\pi m} +1}{e^{2\pi m} - 1} = 1,$ there exists $N_0 > 1,$ such that $1 \leq \frac{e^{2\pi m} +1}{e^{2\pi m} - 1} < 2$ for all $m \geq N_0,$ and \[\sum_{m, n \geq 1} \frac{1}{(m^2 + n^2)m} = -\frac{\zeta(3)}{2} + \frac{\pi}{2} \sum_{m \geq 1} \left(\frac{e^{2\pi m} +1}{e^{2\pi m} - 1}\right)\frac{1}{m^2} < - \frac{\zeta(3)}{2} + \frac{\pi}{2} \sum_{m = 1}^{N_0 -1} \left(\frac{e^{2\pi m} +1}{e^{2\pi m} - 1}\right)\frac{1}{m^2} + \pi \sum_{m \geq N_0} \frac{1}{m^2}.\] Since each term on the right-hand side above converges, we conclude that $\sum_{m, n \geq 1} \frac{1}{(m^2 + n^2)m}$ converges, and this concludes the argument.

\end{proof}

\begin{lem}
\label{lem:modular}
We keep the same notation as in Theorem \ref{thm:modular}. Let $q_{k} = e^{-\pi t_{k}} $. Then 
\[
c_0(k) = \frac{16}{\pi}\sum_{\substack{d, f \geq 1\\ f\textnormal{ odd}}} \frac{d}{f} (-1)^{d+1} \chi_{-4}(f) q_{k}^{df},
\]
where $\chi_{-4}(n) = \left(\frac{-4}{n} \right)$ is the odd quadratic character modulo $4$.

\end{lem}
\begin{proof}
Using \cite[p. 437]{StraubZudilin}, we find that $F(t)$ has the following modular parametrisation:
\[
F(x(\tau)) = \frac{2 \tau}{ \pi i} \frac{\eta(4 \tau)^{10}}{ \eta(2 \tau)^4 \eta(8 \tau)^4}.
\]
Hence $F(x(\tau))/\tau$ is a modular form of weight $1$ for the group $\Gamma_1(8)$.
We start by making a change of variables $t = x(\tau)$ in the integral defining $c_0(k)$. Since $x(i/(4 t_k)) = k$, we find
\[
c_0(k) = \int_k^4 F(t) \, \dd t = \int_{i/(4t_k)}^{0} F(x(\tau)) \dd x(\tau) =  \int_{i/(4t_k)}^{0} \frac{2 \tau}{ \pi i} \frac{\eta(4 \tau)^{10}}{ \eta(2 \tau)^4 \eta(8 \tau)^4}  \dd x(\tau).
\]
Using \cite[Table 2]{AyginToh}, it follows that the derivative $\frac{\dd x(\tau)}{\dd \tau}$ is again an eta-quotient.
\[
\dd x(\tau) = 2 \pi i \cdot 16 \left( \frac{\eta(2 \tau)^3 \eta(8 \tau)^2}{\eta(4 \tau)^4} \right)^4 \dd \tau = 2 \pi i \left(16q - 192q^3 + 1120q^5 - 4480q^7 + 14544q^9 + \cdots \right)\dd \tau.
\]
Thus, we arrive at
\begin{align*}
    \int_k^4 F(t) \, \dd t &= 64 \int_{i/(4t_k)}^{0} \tau \cdot \left( \frac{\eta(2 \tau)^4 \eta(8 \tau)^2}{\eta(4 \tau)^3} \right)^2 \dd \tau.
\end{align*}
Next, we apply the transformation $\tau = -1/(8 \tau')$. From the functional equation of the eta function \cite{Siegel}: $\eta(-1/
(ix)) = \sqrt{x} \eta(i x)$ for $x > 0$, we find that
\[
\left( \frac{\eta(2 \tau)^4 \eta(8 \tau)^2}{\eta(4 \tau)^3} \right)^2 = 32 i (\tau')^3 \left(\frac{\eta(4 \tau')^4 \eta(\tau')^2}{\eta(2\tau')^3} \right)^{2}.
\]
This implies that
\[
64 \int_{i/(4t_k)}^{0} \tau \cdot \left( \frac{\eta(2 \tau)^4 \eta(8 \tau)^2}{\eta(4 \tau)^3} \right)^2 \dd \tau = - 32 i \int_{i t_k/2}^{+ i \infty} \left(\frac{\eta(4 \tau')^4 \eta(\tau')^2}{\eta(2\tau')^3} \right)^{2} \, \dd \tau'.
\]
The integrand can be written as an Eisenstein series (see \cite[A138501]{oeis}):
\[
\left(\frac{\eta(4 \tau)^4 \eta(\tau)^2}{\eta(2\tau)^3} \right)^{2} = \sum_{n \geq 1} (-1)^n \sum_{d | n} d^2 (-1)^{n/d} \chi_{-4}(n/d) q^n.
\]
We can rewrite the $q$-series in the following way
\begin{align*}
    \sum_{n \geq 1} (-1)^n \sum_{d | n} d^2 (-1)^{n/d} \chi_{-4}(n/d) q^n &= \sum_{d, f \geq 1} d^2 (-1)^{fd + f} \chi_{-4}(f) q^{fd}\\
    &= \sum_{\substack{d, f \geq 1\\ f\textnormal{ odd}}} d^2 (-1)^{d+1} \chi_{-4}(f) q^{fd}.
\end{align*}
Hence, 
\begin{align*}
    \int_{k}^4 F(t) \, \dd t &= -32 i \int_{i t_k/2}^{+ i \infty} \left(\frac{\eta(4 \tau)^4 \eta(\tau)^2}{\eta(2\tau)^3} \right)^{2} \, \dd \tau\\
    &= \frac{32}{2 \pi} \int_0^{q_{k'}} \left(\frac{\eta(4 \tau)^4 \eta(\tau)^2}{\eta(2\tau)^3} \right)^{2} \, \frac{\dd q}{q}\\
        &= \frac{16}{\pi}\sum_{\substack{d, f \geq 1\\ f\textnormal{ odd}}} d^2 (-1)^{d+1} \chi_{-4}(f) \frac{q_{k}^{fd}}{fd}\\
        &= \frac{16}{\pi}\sum_{\substack{d, f \geq 1\\ f\textnormal{ odd}}} \frac{d}{f} (-1)^{d+1} \chi_{-4}(f) q_{k}^{fd},
\end{align*}
where we have used that $\dd q = 2 \pi i q \dd \tau$. This is precisely what we needed to show.
\end{proof}

\begin{proof}[Proof of Theorem \ref{thm:modular}]
By Lemma \ref{Lem:absconv} the sum converges absolutely, so we can choose the order of summation. We first sum over $m$. We use the following identity, see \cite[p. 19]{Weil}: for $\tau$ in the complex upper half-plane and $q = e^{2 \pi i \tau}$, we have
\begin{equation}
\label{eq:cotangentsum}
\frac{1}{(2 \pi i)^2}  \sum_{n \in \Z} \frac{1}{(\tau + n)^2} = \sum_{r \geq 1} r q^{r}.    
\end{equation}

It follows that for fixed $m \neq 0$
\begin{align*}
\sum_{\substack{n \in \Z \\ n \textnormal{ odd}}}\frac{1}{(2 m \tau + n)^2} &= \left( \sum_{\substack{n \in \Z_{\neq 0}}}- \sum_{\substack{n \in \Z_{\neq 0} \\ n \textnormal{ even}}} \right)  \frac{1}{(2 m \tau + n)^2}\\
&= \sum_{\substack{n \in \Z_{\neq 0}}} \frac{1}{(2 m \tau + n)^2} - \frac{1}{4} \sum_{\substack{n \in \Z_{\neq 0}}} \frac{1}{(m \tau + n)^2}\\
&= (2 \pi i)^2 \sum_{r \geq 1} \left( r q^{2 mr} - \frac{1}{4} r q^{mr} \right),
\end{align*}
where the last line follows from \eqref{eq:cotangentsum}. Next, we sum over $m$, splitting the sum into two cases: $m \geq 1$ and $m \leq -1$. Since 
\[
\sum_{\substack{n \in \Z \\ n \textnormal{ odd}}}\frac{1}{(2 m \tau + n)^2}
\]
depends only on the absolute value of $m$, we can write the sum over $m\leq -1$ as
\[
\sum_{\substack{m \leq -1 \\ m \textnormal{ odd}}} \frac{e^{-2 \pi i m/4 }}{m} \sum_{\substack{n \in \Z \\ n \textnormal{ odd}}}\frac{1}{(2 m \tau + n)^2} = -\sum_{\substack{m \geq 1 \\ m \textnormal{ odd}}} \frac{e^{2 \pi i m/4 }}{m} \sum_{\substack{n \in \Z \\ n \textnormal{ odd}}}\frac{1}{(2 m \tau + n)^2}.
\]
We conclude that
\begin{align*}
\frac{8 i}{\pi^3} \sum_{\substack{m, n \in \Z \\ m, n \textnormal{ odd}}}\frac{e^{-2 \pi i m/4}}{(2 m\tau + n)^2 m} &= \frac{-32 i}{\pi} \sum_{\substack{m \geq 1, r \geq 1 \\ m \textnormal{ odd}}} \frac{e^{-2 \pi i m /4}-e^{2 \pi i m /4}}{m} \left(r q^{2 m r} - \frac{1}{4} r q^{mr} \right) \\
&= \frac{-64}{\pi} \sum_{\substack{m \geq 1, r \geq 1 \\ m \textnormal{ odd}}} \frac{\chi_{-4}(m)}{m} \left(r q^{2 m r} - \frac{1}{4} r q^{mr} \right) \\
&= \frac{16}{\pi} \sum_{\substack{m \geq 1, r \geq 1 \\ m \textnormal{ odd}}} \frac{r}{m} (-1)^{r+1} \chi_{-4}(m) q^{mr}.
\end{align*}
If we now apply this identity to $\tau = \tau_k/2$, we find, using Lemma \ref{lem:modular}
\[
c_0(k) = \frac{8 i}{\pi^3} \sum_{\substack{m, n \in \Z \\ m, n \textnormal{ odd}}}\frac{e^{-2 \pi i m/4}}{(m\tau_k + n)^2 m}.
\]
This concludes the proof of our statement.

\end{proof}


\bibliographystyle{alpha}
\bibliography{Bibliography}
\end{document}